\documentclass[english]{amsart}

\usepackage{amsmath,amssymb,enumerate}
\usepackage{amscd,amsxtra,calc}

\usepackage[T1]{fontenc}
\usepackage[all]{xy}

\usepackage{babel}
\usepackage{amstext}
\usepackage{amsmath}
\usepackage{amsfonts}
\usepackage{latexsym}
\usepackage{ifthen}
\usepackage{verbatim}
\usepackage{enumerate}

\usepackage{xypic}
\xyoption{all}
\newcommand{\Id}{{\rm Id}}
\newcommand{\quo}{{\rm quo}}

\newcommand{\rk}{{\rm rk}}
\newcommand{\codim}{{\rm codim}}

\newcommand{\Bs}{{\rm Bs}}

\newcommand{\pr}{{\rm pr}}

\newtheorem{lemma1}{}[section]

\newenvironment{lemma}{\begin{lemma1}{\bf Lemma.}}{\end{lemma1}}

\newenvironment{theorem}{\begin{lemma1}{\bf Theorem.}}{\end{lemma1}}
\newenvironment{proposition}{\begin{lemma1}{\bf Proposition.}}{\end{lemma1}}
\newenvironment{corollary}{\begin{lemma1}{\bf Corollary.}}{\end{lemma1}}
\newenvironment{remark}{\begin{lemma1}{\bf Remark.}\rm}{\end{lemma1}}
\newenvironment{definition}{\begin{lemma1}{\bf Definition.}}{\end{lemma1}}

\newenvironment{setup}{\begin{lemma1}{\bf Setup.}}{\end{lemma1}}

\newenvironment{conjecture}{\begin {lemma1}{\bf Conjecture.}}{\end{lemma1}}

\newenvironment{remark*}{{\bf Remark.}}{}
\newenvironment{example*}{{\bf Example.}}{}

\newcommand{\Q}{\ensuremath{\mathbb{Q}}}

\newcommand{\C}{\ensuremath{\mathbb{C}}}
\newcommand{\N}{\ensuremath{\mathbb{N}}}
\newcommand{\PP}{\ensuremath{\mathbb{P}}}

\newcommand{\holom}[3]{\ensuremath{#1:#2  \rightarrow #3}}
\newcommand{\fibre}[2]{\ensuremath{#1^{-1} (#2)}}

\makeatletter
\ifnum\@ptsize=0  \fi \ifnum\@ptsize=2  \fi 

\newcommand\sE{{\mathcal E}}

\newcommand\sF{{\mathcal F}}

\newcommand\sI{{\mathcal I}}

\newcommand\sO{{\mathcal O}}

\newcommand\sV{{\mathcal V}}

\newcommand\bQ{{\mathbb Q}}
\newcommand\bN{{\mathbb N}}

\newcommand{\chow}[1]{\ensuremath{\mathcal{C}(#1)}}

\newcommand{\ep}[0]{{\epsilon}}

\newcommand{\supp}[0]{\operatorname{Supp}}

\newcommand{\Alb}[0]{\operatorname{Alb}} 
\newcommand{\End}[0]{\operatorname{End}}

\newcommand{\Univ}[0]{\mathcal U}

\title{A decomposition theorem for projective manifolds with nef anticanonical bundle}
\date{June 27, 2017}


\author{Junyan Cao}
\author {Andreas H\"oring}

\address{Andreas H\"oring, Universit\'e C\^ote d'Azur, CNRS, LJAD, France}
\email{Andreas.Hoering@unice.fr}

\address{Junyan Cao, Universit\'e Paris 6 \\
Institut de Math\'{e}matiques de Jussieu\\
4, Place Jussieu, Paris 75252, France }
\email{junyan.cao@imj-prg.fr}

\begin{document}

\begin{abstract} 
Let $X$ be a simply connected projective manifold with nef anticanonical bundle.
We prove that $X$ is a product of a rationally connected manifold and a manifold with 
trivial canonical bundle. As an application we describe the MRC fibration of
any projective manifold with nef anticanonical bundle. 
\end{abstract}

\maketitle


\section{Introduction}

\subsection{Main results}

Manifolds with nef anticanonical class naturally appear as an interpolation of Fano manifolds
and compact K\"ahler manifolds with trivial canonical class, but they come with many new features: 
in general the anticanonical bundle is not semiample (blow-up $\PP^2$ in nine general points),
nor hermitian semipositive \cite[Ex.1.7]{DPS94}.
Initiated by the fundamental papers of Demailly, Peternell and Schneider \cite{DPS93, DPS94, DPS96},
a central goal of the theory is to understand the natural maps attached to $X$, e.g. the Albanese map or the MRC-fibration.
The expected outcome of this study is contained in the following conjecture which should be understood
as an analogue of the Beauville-Bogomolov decomposition theorem \cite{Bea83}:

\begin{conjecture}\label{newcon}
Let $X$ be a compact K\"{a}hler manifold with nef anticanonical class. 
Then the universal cover $\widetilde{X}$ of $X$ decomposes as a product
$$
\widetilde{X} \simeq \mathbb{C}^q \times \prod Y_j \times \prod S_k \times Z,
$$
where $Y_j$ are irreducible Calabi-Yau manifolds, $S_k$ are irreducible hyperk\"{a}hler manifolds,
and $Z$ is a rationally connected manifold.
\end{conjecture}

This conjecture has been proven under the 
stronger assumption that $T_X$ is nef \cite{CP91, DPS94}, $-K_X$ is hermitian semipositive \cite{DPS96, CDP15} 
or the general fibre of the Albanese map is weak Fano \cite{CH17}.

In this paper we focus on the case where $X$ is a projective manifold. 
Very recently the first-named author proved that for any projective manifold $X$ with nef anticanonical bundle
the Albanese map $X \rightarrow \mbox{Alb}(X)$ is a locally trivial fibration \cite{Cao16}.
We know that the fundamental group is almost abelian \cite{Pau97}, so the next step is to study $X$ when it is simply connected. 
We show that the structure of $X$ is as simple as possible:

\begin{theorem} \label{theoremmain}
Let $X$ be a projective manifold such that $-K_X$ is nef and $\pi_1(X) = 1$.
Then $X \simeq Y \times F$ such that $K_Y \sim 0$ and $F$ is a rationally connected manifold.
\end{theorem}

By \cite[Cor.1.4]{Cao16} this immediately implies:

\begin{corollary}
Conjecture \ref{newcon} holds if $X$ is projective.
\end{corollary}

Moreover we obtain a precise description of the MRC-fibration:

\begin{theorem} \label{theoremMRC}
Let $X$ be a projective manifold such that $-K_X$ is nef. Then there exists a locally trivial fibration
$$
X \rightarrow Y
$$
such that the fibre $F$ is rationally connected and $K_Y \equiv 0$. 
\end{theorem}

Using Mori theory and explicit computations this statement was shown for threefolds in \cite[Sect.3]{BP04}.
For K\"ahler manifolds even this low-dimensional case is still open.

\subsection{Strategy and organisation of the paper}
By a result of Qi Zhang \cite{Zha05} the base of 
the MRC fibration has Kodaira dimension zero, so the situation 
of Theorem \ref{theoremmain} looks similar to the case of the Albanese fibration
studied in \cite{Cao16}. However, even by using techniques from the MMP, it seems a priori not clear that the MRC-fibration
can be represented by a {\em holomorphic} map $X \rightarrow Y$. Therefore we have to proceed in a less direct way:
the MRC-fibration is an almost holomorphic map, so it determines a unique component of the Chow space parametrising
its general fibres. We denote by $Y$ a resolution of this component, and by $\varphi: \Gamma \rightarrow Y$ a resolution 
of the universal family over this component. 

The natural map $\pi: \Gamma \rightarrow X$ is birational with exceptional locus $E$, and our first goal is to describe the structure of the fibre space $\varphi: \Gamma \rightarrow Y$.  
The tool for this description is the positivity of certain direct image sheaves: for some sufficiently ample line bundle $A$
on $X$, we set $L:= \pi^\star A$. The main technical result of this paper (Proposition \ref{propositiontrivial}) is that for some $m \gg 0$ and all $p$ sufficiently divisible, the direct image sheaves
$$
\sV_p := \varphi_\star (\sO_\Gamma(p L + p m E - \frac{p}{r} \varphi^\star \det \varphi_\star 
(\sO_\Gamma(L + m E))))
$$ 
are trivial vector bundles over a certain open subset $Y_0 \subset Y$ which is simply connected and has only constant
holomorphic functions (cf. Setup \ref{setup} for the definition of $Y_0$), where $r$ is the rank of $\varphi_\star 
(\sO_\Gamma(L + m E))$.
Following an argument going back to \cite{DPS94} this implies that we have a birational map
$$
\fibre{\varphi}{Y_0} \dashrightarrow Y_0 \times F,
$$
where $F$ is a general fibre of the MRC fibration. 

The second step is to see how the product structure of some birational model yields a product structure on $X$. This  
typically requires some control of the birational map $X \dashrightarrow Y \times F$, in our case we simply observe
that the product structure on $Y_0 \times F$ induces a splitting of the tangent bundle of 
$\fibre{\varphi}{Y_0} \setminus E$. Since (by the construction of $Y_0$) the complement of 
$\fibre{\varphi}{Y_0} \setminus E \subset X$ has codimension at least two, we obtain a splitting of the tangent bundle $T_X$
defining two algebraically integrable foliations. Now standard arguments for manifolds with split tangent bundle yield the theorem.

The technical core of this paper is to study the direct image sheaves on the non-compact quasi-projective set $Y_0 \subset Y$,
making only assumptions of a birational nature on the (relative) anticanonical divisor. This setup has not been studied in earlier papers, we therefore establish some basic results in Section \ref{sectionpositivity} (and the appendix) before applying
them to our situation in Subsection \ref{sectiontriviality}.

{\bf Acknowledgements.} It is a pleasure to thank S. Boucksom, J.-P. Demailly and M. P\u aun for very helpful discussions.
In particular, Lemma \ref{curvestablecase} comes from a nice idea of J.-P. Demailly. A part of the article was done during
the first-named author's visit to the Institut for Mathematical Sciences
of National University of Singapore for the program "Complex Geometry, Dynamical Systems and Foliation Theory". 
He is very grateful to the organizers for the invitation and the support.
This work was partially supported by the Agence Nationale de la Recherche grant project Foliage\footnote{ANR-16-CE40-0008} and the Agence Nationale de la Recherche grant ``Convergence de Gromov-Hausdorff en g\'{e}om\'{e}trie k\"{a}hl\'{e}rienne"
(ANR-GRACK).

\begin{center}
{\bf Notation and basic results}
\end{center}

For general definitions we refer to \cite{Har77, Kau83, Dem12}. We use the terminology of \cite{Deb01}
for birational geometry and \cite{Laz04a} for (algebraic) notions of positivity.
Manifolds and varieties will always be supposed to be irreducible.
A fibration is a proper surjective map with connected fibers \holom{\varphi}{X}{Y} between normal complex spaces.

Let us recall some basic facts about reflexive sheaves:

\begin{proposition} \label{facts}
\begin{enumerate}[1)]
\item \cite[Cor.1.4]{Har80} Let $M$ be a manifold, and let $\sF$ be a reflexive sheaf on $M$. Then there exists a subset $Z$ of codimension at least three
such that $\sF \otimes \sO_{M \setminus Z}$ is locally free. 
\item \cite[Cor.1.7]{Har80} Let \holom{\varphi}{M}{N} be a proper, equidimensional, dominant morphism between normal varieties. If $\sF$ is a reflexive sheaf on $M$, then $\varphi_\star (\sF)$ is a reflexive sheaf on $N$. 
\item \cite[Prop.1.6]{Har80} Let $M$ be a normal variety, and let $\sF_1$ and $\sF_2$ be reflexive sheaves on $M$. Suppose that there exists
a closed subset $Z \subset M$ of codimension at least two such that 
$\sF_1 \otimes \sO_{M \setminus Z} \simeq \sF_2 \otimes \sO_{M \setminus Z}$.
Then we have $\sF_1 \simeq \sF_2$. 
\end{enumerate}
\end{proposition}

\section{Positivity of vector bundles} \label{sectionpositivity}

\subsection{Preliminaries}

We first recall the definition of possibly singular Hermitian metrics on vector bundles and some basic properties.
We refer to \cite{BP08,PT14,Rau15,Pau16,HPS16} for more details.
Let $E\rightarrow X$ be a holomorphic vector bundle of rank $r$ on a complex manifold $X$. We denote by 
$$H_r := \{A =(a_{i, \overline{j}})\}$$
the set of $r\times r$, semi-positive definite Hermitian matrices. Let $\overline{H}_r$ be the space of semi-positive, possibly unbounded Hermitian
forms on $\mathbb{C}^r$. A possibly singular Hermitian metric $h$ on $E$ is locally given by a measurable map with values in $\overline{H}_r$
such that 
\begin{equation}\label{genefinit}
0 < \det h < +\infty  \qquad\text{ almost everywhere.}
\end{equation}
In the above definition, a matrix valued function $h= (h_{i, \overline{j}})$ is measurable provided that all entries 
$h_{i, \overline{j}}$ are measurable.

\begin{definition}
Let $X$ be a (not necessarily compact) complex manifold, and let $(E, h)$ be a holomorphic vector bundle on $X$ with a possibly singular hermitian metric $h$. 
Let $\theta$ be a smooth closed $(1,1)$-form on $X$.

We say that $i\Theta_h (E) \succcurlyeq \theta \otimes \Id_{\End (E)}$, if for any open set $U\subset X$ and $s\in H^0 (U, E^\star)$, 
$$
dd^c \ln |s|_{h^\star} - \theta \geq 0 \qquad\text{on } U \text{ in the sense of currents}.
$$
We say that $(E,h)$ is positively curved on $X$ if $\Theta_h (E) \succcurlyeq 0$. 
\end{definition}

We introduce here a weaker notion of positivity which will be important for us.

\begin{definition}\label{weakpocu}
Let $X$ be a K\"ahler manifold with a K\"ahler metric 
$\omega_X$. Let $E$ be a holomorphic vector bundle on $X$.
We say that $E$ is weakly positively curved on $X$, if for every $\epsilon >0$, there exists a possibly singular hermitian metric $h_\ep$ on $E$ such that 
$i \Theta_{h_\ep} (E) \succcurlyeq  -\ep \omega_X \otimes \Id_{\End (E)}$. 
\end{definition}

\begin{remark}\label{equdefn}

{\em (i):} If $i\Theta_h (E) \succcurlyeq \theta \otimes \Id_{\End (E)}$ for some smooth form $\theta$, by the definition of quasi-psh functions, 
for every subset $U \Subset X$, we have $h \geq c\cdot \Id$ on $U$ for some constant $c >0$ depending on $U$.

{\em (ii):} If $X$ is projective, by using the same argument as in \cite[Sect 2]{PT14}\cite[Thm 2.21]{Pau16}, 
we can prove that the notion of "weakly positively curved vector bundle"  implies the weak positivity in the sense of Viehweg.
However, the inverse seems unclear.
\end{remark}

Now we recall two basic properties of possibly singular hermitian metrics 
proved in \cite{PT14, Pau16} which are extremely useful.

\begin{proposition}\label{speccase}
Let $E$ be a holomorphic vector bundle on a  complex manifold $X$ and let $\theta$ be a smooth $(1,1)$-form on $X$.
Let $X_0 \subset X$ be a Zariski open set and let $(E|_{X_0},h_0)$ be a possibly singular hermitian metric $h_0$ defined only on $E |_{X_0}$ such that 
$i\Theta_{h_0} (V) \succcurlyeq \theta \otimes \Id_{\End (E)}$ on $X_0$.

{\em (i):} If  $h_0 \geq c \cdot \Id$ on $X_0$ for some constant $c >0$, 
then  $h_0$ can be extended as a possibly singular hermitian metric $h$ on the total space $E$ over $X$ 
such that
$$i\Theta_{h} (E) \succcurlyeq \theta \otimes \Id_{\End (E)} \qquad\text{on }X .$$

{\em (ii):} If $\codim_X (X\setminus X_0) \geq 2$,  then  $h_0$ can be extended as a possibly singular hermitian metric $h$ on the total space $E$ over $X$
such that
$$i\Theta_{h} (E) \succcurlyeq \theta \otimes \Id_{\End (E)} \qquad\text{on }X .$$
\end{proposition}

\begin{proof}
Let $s\in H^0 (U, E^\star)$ for some open set $U\subset X$. By assumption, $dd^c \ln |s|_{h_0^\star} \geq \theta$ on $U \cap X_0$. 
The condition  $h_0 \geq c \cdot \Id$ implies that $\ln |s|_{h_0^\star}$ is  upper bounded on any compact set
in $U$. By Hartogs, $\ln |s|_{h_0^\star}$ can be extended as a quasi-psh function on $U$ with $dd^c \ln |s|_{h_0^\star} \geq \theta$ on $U$.
This proves $(i)$. As for $(ii)$, we can directly apply Hartogs' theorem.
\end{proof}

As a consequence of Proposition \ref{speccase}, we have the following important \textit{curvature increasing} property.

\begin{proposition}\label{quotientweakpositive}\cite{PT14}\cite[Lemma 2.19]{Pau16}
Let $\pi: E_1 \rightarrow E_2$ be a generically surjective morphism between two holomorphic vector bundles on a complex manifold $X$.
Let $h$ be a possibly singular metric on $E_1$ such that $i\Theta_h (E_1) \succeq \theta \otimes \Id_{\End (E_1)}$ 
for some smooth $(1,1)$-form $\theta$ on $X$.
Then $h$ induces a quotient metric $h_\quo$ on $E_2$ such that $i\Theta_{h_\quo} (E_2) \succeq \theta \otimes \Id_{\End (E_2)}$ on $X$.

In particular, if $E_1$ is (resp. weakly) positively curved, $E_2$ is also (resp. weakly) positively curved.
\end{proposition}

\begin{proof}
Let $X_0 \subset X$ be the Zariski open subset such that $\pi$ is surjective over $X_0$. Then $h$ induces a quotient metric $h_2$ on $E_2|_{X_0}$ such 
that $i\Theta_{h_2} (E_2) \succeq \theta \otimes \Id_{\End (E_2)}$ on $X_0$. Let $x\in X\setminus X_0$ be a generic point. Thanks to Remark \ref{equdefn},
there exists a small neighborhood $U$ of $x$
such that $h \geq c \cdot \Id$ on $U$ for some constant $c >0$.
Therefore $h_2 \geq c \cdot \Id$ on $U \cap X_0$. 
By applying Proposition \ref{speccase}, we can know that $h_2$ can be extended as a metric $h_\quo$ on the total space $E_2$ 
such that $i\Theta_{h_\quo} (E_2) \succeq \theta \otimes \Id_{\End (E)}$ on $X$.
\end{proof}

Thanks to the fundamental works like \cite{Ber09, BP08, PT14} among many others,
we know that the notion of possibly singular hermitian metrics fits well with the positivity of direct images. We will use the following version in the article.

\begin{theorem} \cite[Section 5]{PT14}\label{mainposit}
Let $f: X\rightarrow Y$ be a fibration between two projective manifolds and let $L$ be a line bundle on $X$ with a possibly singular metric $h_L$
such that $\Theta_{h_L} (L) \geq 0$. Let $m\in \mathbb{N}$ such that the multiplier ideal sheaf $\mathcal{I} (h_L ^{\frac{1}{m}} |_{X_y})$
is trivial over a generic fiber $X_y$, namely $\int_{X_y} |e_L|_{h_L} ^{\frac{2}{m}} < +\infty$, where $e_L$ is a basis of $L$. 

Let $Y_1$ be the locally free locus of $f_\star (\mathcal{O}_X (m K_{X/Y} +L) )$. Then the vector bundle $f_\star (\sO_X(m K_{X/Y} +L)) \otimes \sO_{Y_1}$ admits a possibly singular hermitian metric $h$ such that $(f_\star (\sO_X(m K_{X/Y} +L)) \otimes \sO_{Y_1} , h)$  
is positively curved on $Y_1$. 
\end{theorem}

\begin{remark}
We recall briefly the idea of the proof. In fact, by using \cite{Ber09, BP08}, over some Zariski open set $Y_0$ of $Y$, 
we can construct a continuous metric $h$ on $f_\star (\sO_X(m K_{X/Y} +L)) \otimes \sO_{Y_0}$
such that it is positively curved. Thanks to the Ohsawa-Takegoshi extension theorem and the semistable reduction, we know that $h \geq c \cdot \Id$ for some constant $c >0$. 
The theorem is thus proved by using Proposition \ref{speccase},
\end{remark}

As a consequence of Theorem \ref{mainposit} and an easy trick \cite[Lemma 5.25]{CP17},
we have the following variant.
\begin{proposition}\label{lowercurcontrol}
Let $f: X\rightarrow Y$ be a fibration between two projective manifolds and let $L$ be a line bundle on $X$ with a possibly singular metric $h_L$
such that $\Theta_{h_L} (L) \geq f^\star\theta$ for some smooth $d$-closed $(1,1)$-form $\theta$ (not necessarily positive) on $Y$. 
Let $m\in \mathbb{N}$ such that $\mathcal{I} (h_L ^{\frac{1}{m}} |_{X_y}) = \mathcal{O}_{X_y}$ for a generic fiber $X_y$.

Let $Y_1$ be the locally free locus of $f_\star (\mathcal{O}_X (m K_{X/Y} +L) )$.
Then the vector bundle $f_\star (\mathcal{O}_X (m K_{X/Y} +L) ) \otimes \mathcal{O}_{Y_1}$ 
admits a possibly singular hermitian metric $h$ such that 
$$i \Theta_h (f_\star (\mathcal{O}_X (m K_{X/Y} +L) )) \succcurlyeq \theta \otimes \Id_{\End (f_\star (\mathcal{O}_X (m K_{X/Y} +L) ))} \qquad\text{on } Y_1 .$$
\end{proposition}

\bigskip

\subsection{Numerical flatness}
Let $X$ be a compact K\"ahler manifold and let $E$ be a holomorphic vector bundle on $X$.  
We say that $E$ is numerically flat if $E$ and its dual $E^\star$ are nef (in the sense of \cite{DPS94}),
one sees easily that this is equivalent to
$$
E \ \mbox{is nef and } c_1 (E) =0 \in H^{1,1} (X, \mathbb{Q}).
$$
The aim of this subsection is to give a sufficient condition for a vector bundle $E$ with $c_1 (E) =0$
to be numerically flat.

\begin{lemma}\label{curvestablecase}
Let $C$ be a smooth curve and let $E$ be a numerically flat vector bundle of rank $n$ on $C$.
Let $\pi : \mathbb{P} (E) \rightarrow C$ be the natural projection. 
Fix an ample line bundle $A$ over $C$. 
Let $T_m \geq 0$ be a positive current with analytic singularities in the class of 
$c_1 (\mathcal{O}_{\mathbb{P} (E)} (m) + \pi^\star A )$ such that $T_m$ is smooth in the neighbourhood of a general $\pi$-fibre.
Set 
$$a_m := \max_{x\in \mathbb{P} (E)} \nu (T_m ,x) ,$$ 
where $\nu (T_m, x)$ is the Lelong number of $T_m$ over $x$.
Then
\begin{equation}\label{lim0sup}
\lim_{m \rightarrow 0}\frac{a_m}{m} = 0 .
\end{equation}
\end{lemma}

\begin{proof}
Thanks to Proposition \ref{lelongnumcont}, 
for the current $T_m$,  we can find a $k \in \{1, 2, \cdots , n\}$ depending on $m$, a closed $(k ,k)$-positive current $\Theta_{k}$ in 
the same class of $(c_1 (\mathcal{O}_{\mathbb{P} (E)} (m) + \pi^\star A))^{k}$ and a subspace $Z_m$
of codimension $k$ contained in some $\pi$-fibers 
such that 
\begin{equation}\label{ineqimportnew}
\Theta_{k} \geq (\frac{a_m}{n})^{k}  \cdot [Z_m] 
\end{equation}
in the sense of currents.  Set $\alpha :=c_1 (\mathcal{O}_{\mathbb{P} (E)} (1))$, then $\alpha$ is a nef cohomology class.
Thus \eqref{ineqimportnew} implies 
$$
(m\alpha + \pi^\star c_1 (A))^k \cdot \alpha^{n-k}=\Theta_{k} \cdot \alpha^{n-k} \geq  (\frac{a_m}{n})^{k} \int_{[Z_m]} \alpha^{n-k} .
$$
Note that $Z_m$ is contained in some fiber of $\pi$, we have $\int_{[Z_m]} \alpha^{n-k} \geq 1$. 
Therefore 
\begin{equation}\label{ineqpolynew}
 (m\alpha + \pi^\star c_1 (A))^k \cdot \alpha^{n-k} \geq (\frac{a_m}{n})^{k}
\end{equation}

Since $c_1 (E) =0$, we have $\alpha^{n} =0 \in H^{n,n} (\mathbb{P} (E))$. Therefore
for $m \gg 0$ the left side of \eqref{ineqpolynew} is less or equal than $C_1 m^{k-1} +C_2$ for some uniform constants $C_1$ and $C_2$.
Then \eqref{ineqpolynew} implies that $ (\frac{a_m}{n})^{k} \leq C_1 m^{k-1} +C_2$. As $k\leq n$,  we obtain finally
$$a_m \leq C_3 m^{\frac{n-1}{n}} + C_4 $$
for some uniform constants $C_3$ and $C_4$ independent of $m$. 
This proves \eqref{lim0sup}.
\end{proof}

The following lemma is essentially shown in \cite[Prop 2.3.5]{PT14}\cite[Thm 2.21]{Pau16}.

\begin{lemma}\label{equdef}
Let $X$ be a $n$-dimensional projective manifold and
let $E$ be a holomorphic vector bundle on $X$. 
Let $\pi : \mathbb{P} (E) \rightarrow X$ be the natural projection and let $\omega$ be a K\"ahler metric on $\mathbb{P} (E)$.
Then the following three conditions are equivalent:

{\em (i)}
There exists a very ample line bundle $A$ on $X$ such that the restriction
$$
H^0 (\mathbb{P} (E), \mathcal{O}_{\mathbb{P} (E)}  (m) + \pi^\star A) \rightarrow
 H^0 (\mathbb{P} (E) _x ,\mathcal{O}_{\mathbb{P} (E)}  (m) + \pi^\star A )
$$
is surjective over a generic point $x\in X$ for $m\gg 1$.

{\em (ii)} $\pi_\star(\mathbb{B}_{-} (\mathcal{O}_{\mathbb{P} (E)} (1))) \subsetneq X$.\footnote{Recall that, for a line bundle $L$,  $\mathbb{B}  (L)$ is the stable base locus of $L$, namely $\mathbb{B}  (L) = \cap_{m\geq 1} \Bs (|m L|)$. $\mathbb{B}_{-} (L) := \mathbb{B} (
L +\ep A)$ for any ample $A$ and sufficiently small $\ep >0$.} 

{\em (iii)} For 
every $\ep >0$, there exists a possibly singular metric $h_\ep$ on 
$\mathcal{O}_{\mathbb{P} (E)}  (1)$ such that $i \Theta_{h_\ep} (\mathcal{O}_{\mathbb{P} (E)}  (1)) \geq -\ep \omega$ and the multiplier ideal $\sI ( h_{\ep} ^{\otimes m} |_{\mathbb{P} (E) _x})$ of the restriction of $h_{\ep} ^{\otimes m}$
to the generic fiber $\mathbb{P} (E) _x$ is trivial for every $m\in\mathbb{N}$.
\end{lemma}

\begin{proof}
$(i)\implies (ii)$: It comes from the definition.

$(ii)\implies (iii)$: If $\pi_\star(\mathbb{B}_{-} (\mathcal{O}_{\mathbb{P} (E)} (1))) \subsetneq X$, for every $\ep >0$,  we have
$$\pi_\star (\mathbb{B} (\mathcal{O}_{\mathbb{P} (E)} (1)) +\ep A)) \subsetneq X  ,$$ 
where $A$ is some ample line bundle on $\mathbb{P} (E)$. 
Then for some $m$ large enough, $\pi_\star (\Bs (\mathcal{O}_{\mathbb{P} (E)} (m) +\ep m A)) \subsetneq X$.
Therefore, a basis of $H^0 (\mathbb{P} (E), \mathcal{O}_{\mathbb{P} (E)} (m) +\ep m A) $ induces a metric $h_\ep$ on $\mathcal{O}_{\mathbb{P} (E)} (1)$ with analytic singularities such that
\begin{equation}\label{ineqsing}
i\Theta_{h_\ep} (\mathcal{O}_{\mathbb{P} (E)} (1)) \geq -\ep \omega_A
\end{equation}
and $h_\ep$ is smooth on the generic fiber of $\pi$. $(iii)$ is proved.

$(iii)\implies (i)$: Since $\mathcal{O}_{\mathbb{P} (E)}  (1)$ is relatively ample, there exists some $m_0 \in \mathbb{N}$ and some ample line bundle $A_1$ on $X$
such that $-K_{\mathbb{P} (E)} +\mathcal{O}_{\mathbb{P} (E)}  (m_0) + \pi^\star A_1$ is ample. 
Let $A_2$ be a sufficiently ample line bundle on $X$. For every $m\in\mathbb{N}$ and $\ep$ small enough, 
since $\sI ( h_{\ep} ^{\otimes m} |_{\mathbb{P} (E) _x})$ is trivial for a generic $x\in X$, 
 Ohsawa-Takegoshi extension theorem implies that
$$H^0 (\mathbb{P} (E), K_{\mathbb{P} (E)} + (-K_{\mathbb{P} (E)} +\mathcal{O}_{\mathbb{P} (E)}  (m_0) + \pi^\star A_1) +
\mathcal{O}_{\mathbb{P} (E)}  (m) + \pi^\star A_2) $$
$$\rightarrow
 H^0 (\mathbb{P} (E) _x, K_{\mathbb{P} (E)} + (-K_{\mathbb{P} (E)} +\mathcal{O}_{\mathbb{P} (E)}  (m_0) + \pi^\star A_1) +
\mathcal{O}_{\mathbb{P} (E)}  (m) + \pi^\star A_2) $$
is surjective for a generic $x\in X$. This proves $(i)$.
\medskip
\end{proof}

We can now prove the main result in this subsection:

\begin{proposition}\label{numflatweakpost}
Let $X$ be a $n$-dimensional projective manifold and
let $E$ be a holomorphic vector bundle on $X$ with $c_1 (E)=0$
and $\pi_\star(\mathbb{B}_{-} (\mathcal{O}_{\mathbb{P} (E)} (1))) \subsetneq X$.
Then $E$ is numerically flat.

In particular, let $E$ be a vector bundle on $X$. If $E$ is weakly positively curved on $X$ and $c_1 (E)=0$, then $E$ is numerically flat.
\end{proposition}

\begin{proof}
Let $B$ be an arbitrary curve in $\mathbb{P} (E)$
such that $B \rightarrow \pi(B)$ is a finite morphism. We are done if we prove that 
$$
c_1 (\mathcal{O}_{\mathbb{P} (E)}  (1))\cdot B \geq 0.
$$

{\em Step 1: The current $T_m$.} Let $A$ be a sufficiently ample line bundle on $X$
and $\{s_1 , \cdots , s_{k_m}\}$ be a basis of $H^0 (\mathbb{P} (E), \mathcal{O}_{\mathbb{P} (E)}  (m) + \pi^\star A) $ 
and set $T_m := dd^c \ln \sum\limits_{i=1}^{k_m} |s_i|^2$.
Then $T_m \geq 0$ is a current in the class of $c_1 (\mathcal{O}_{\mathbb{P} (E)}  (m) + \pi^\star A)$.
Set $a_m := \nu (T_m , B)$ for the Lelong number of $T_m$ over a generic point of $B$.
We will prove that 
\begin{equation}\label{lim0}
\lim\limits_{m\rightarrow +\infty}\frac{a_m}{m} =0 , 
\end{equation}

\medskip

Note first that, thanks to Lemma \ref{equdef}, $T_m$ is smooth on the generic fiber of $\pi$.  
Let $C \subset X$ be a curve defined by a complete intersection of general elements of $|A|$ such that
$\pi_\star (B) \cap C$ is a finite, non-empty set. 
Since $C$ is general, the restricted bundle $E |_{C}$ is weakly positively curved and $c_1 (E |_{C}) =0$. 
As $\dim C =1$, the bundle $E |_{C}$ is numerically flat. 

We consider the restriction of $T_m$ on $\pi^{-1} (C)$. Thanks to  Lemma \ref{equdef}, 
$T_m |_{\pi^{-1} (C)}$ is smooth on the fiber $\mathbb{P} (E) _x$, for a generic point $x\in C$.  
We can thus apply Lemma \ref{curvestablecase} to $(E|_C , T_m |_{\pi^{-1} (C)})$ and obtain
\begin{equation}\label{limcur}
\lim\limits_{m\rightarrow +\infty}\frac{\nu (T_m |_{\pi^{-1} (C)}, x)}{m} =0 , 
\end{equation}
for any point $x\in B \cap \pi^{-1} (C)$.
Together with the fact that $\nu (T_m |_{\pi^{-1} (C)}, x) \geq \nu (T_m , B)$ for every $x\in B$
we obtain \eqref{lim0}.

\medskip

{\em Step 2: Final conclusion.}
Fix a K\"ahler metric $\omega$ on $\mathbb{P} (E)$. 
Thanks to \cite[thm 1.1]{Dem92}, we can find a current $\widetilde{T}_m$ in the same class of $T_m$ such that
$$
\widetilde{T}_m \geq - C a_m \cdot \omega \qquad\text{and }\qquad  \nu (\widetilde{T}_m , x ) = \max ( \nu (T_m , x ) -a_m , 0 ) 
$$
for every $x\in \mathbb{P} (E)$,
where $C$ is a uniform constant. In particular, $\nu (\widetilde{T}_m , B ) =0$. Therefore 
$$c_1 ( \mathcal{O}_{\mathbb{P} (E)}  (m) + \pi^\star A )\cdot B =\widetilde{T}_m \cdot B \geq -C a_m \int_{B} \omega .$$
Combining this with \eqref{lim0}, by letting $m\rightarrow +\infty$, we get $c_1 (\mathcal{O}_{\mathbb{P} (E)}  (1) )\cdot B  \geq 0$.

\bigskip

For the second part of the proposition,
if $E$ is weakly positively curved,
for every $\ep >0$, we can find a possible singular metric $h_{\ep,1}$ such that 
$$i\Theta_{h_{\ep,1}} (E) \succcurlyeq -\ep \omega_X \otimes \Id_{\End (E)} .$$ Then $h_{\ep,1}$ induces a possibly singular metric $h_{\ep}$ on 
$\mathcal{O}_{\mathbb{P} (E)}  (1)$ such that $i \Theta_{h_{\ep}} (\mathcal{O}_{\mathbb{P} (E)}  (1)) \geq -\ep \pi^\star \omega_X$.
Moreover, the condition \eqref{genefinit} implies that $h_{\ep}$ is bounded over of generic fiber of $\pi$.
In particular, $\sI ( h_{\ep} ^{\otimes m} |_{\mathbb{P} (E) _x})$ is trivial for every $m\in\mathbb{N}$, where $\mathbb{P} (E) _x$ is a generic fiber.
Thanks to Lemma \ref{equdef}, we have $\pi_\star(\mathbb{B}_{-} (\mathcal{O}_{\mathbb{P} (E)} (1))) \subsetneq X$.
Therefore $E$ is numerically flat.
\end{proof}

As a corollary, we obtain

\begin{corollary}\label{trivialexten}
Let $S$ be a smooth projective surface, and let $Z \subset S$ be a finite set.
Let $V$ be a vector bundle on $S \setminus Z$ such that $c_1 (V)=0$ and $V$ is weakly positively curved on $S \setminus Z$.
Denote by $j: S \setminus Z \rightarrow S$ the inclusion and set $\widetilde V:=(j_\star V)^{\star\star}$.

Then $\widetilde{V}$ is a numerically flat vector bundle.
\end{corollary}

\begin{proof}
The sheaf $\widetilde V$ is reflexive, hence locally free on the surface $S$. 
Moreover $c_1 (\widetilde{V})=0$, since $Z$ has codimension two.
By Proposition \ref{speccase}, $\widetilde{V}$ is also weakly positively curved on $S$.
By Proposition \ref{numflatweakpost} it is thus numerically flat.
\end{proof}

\section{Simply connected projective manifolds with nef anticanonical bundle}

This whole section is devoted to the proof of Theorem \ref{theoremmain}. After fixing the notation for the whole section,
we establish some basic geometric properties of the `universal family' $\Gamma \rightarrow Y$ in
Subsection \ref{subsectionbirational}. The results and proofs are already contained in \cite{Zha05}, but we give the details
for the convenience of the reader. In Subsection \ref{sectiontriviality} contains the proof of the key technical result
on direct image sheaves. Finally, in Subsection \ref{subsectionproof} we use these results to describe first the birational
structure of $\Gamma \rightarrow Y$ and translate it later into a biregular statement for $X$.

\begin{setup} \label{setup}
Let $X$ be a uniruled projective manifold such that $-K_X$ nef and $\pi_1(X)=1$. 
There exists a unique irreducible component of the Chow space $\chow{X}$
such that a very general point corresponds to a very general fibre of the MRC fibration.
We denote by $Y$ a resolution of singularities of this irreducible component,
and by $\holom{f}{X'}{Y}$ the normalisation of the pull-back of the 
universal family over it. Let $\holom{p}{X'}{X}$ be the natural map,
then $p$ is birational and, since the MRC fibration is almost holomorphic, the $p$-exceptional locus does not dominate $Y$.

By \cite[Cor.1]{Zha05} the Kodaira dimension $\kappa(Y)$ is equal to zero, and we denote by
$N_Y$ an effective $\Q$-Cartier divisor on $Y$ such that $N_Y \sim_\Q K_Y$.

Let $A$ be a very ample line bundle on $X$. Up to replacing $A$ by some multiple,
we can suppose that
\begin{equation}\label{surgene}
\mbox{\rm Sym}^p H^0 (X_y , A) \rightarrow H^0 (X_y , p A)
\end{equation}
is surjective for every $p\in\mathbb{N}$ where $y\in Y$ is a general point.

Since $X'$ is obtained by base change from a universal family of cycles, the pull-back $p^\star A$ is globally generated and ample on every fibre of $f$.

Choose now a resolution of singularities $\holom{\mu}{\Gamma}{X'}$
and denote by $\holom{\varphi}{\Gamma}{Y}$ and $\holom{\pi}{\Gamma}{X}$
the induced maps. Set $L :=\pi^\star A$ and let $r$ be the rank of $\varphi_\star( \mathcal{O}_\Gamma (L))$.
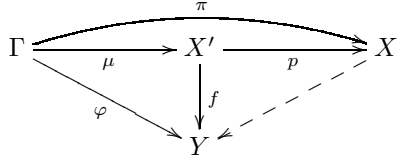
$$
\xymatrix{
\Gamma \ar[rrd]_\varphi  \ar[rr]_\mu \ar @/^1pc/[rrrr]^\pi & & X' \ar[rr]_p \ar[d]^f & & X \ar@{-->}[lld]
\\
& & Y & &}
$$ 
Since $X$ is $\Q$-factorial,
we know that the $\pi$-exceptional locus has pure codimension one, and we denote it by $E$.

Let $Y_0 \subset Y$ be the largest Zariski open subset such that
$\Gamma_0 := \fibre{\varphi}{Y_0} \rightarrow Y_0$ is equidimensional and for every prime divisor $B \subset Y_0$, the pull-back $\varphi^\star B$ is not contained in the $\pi$-exceptional locus $E$. Denote by $P \subset Y$ the largest reduced divisor that is contained
in $Y \setminus Y_0$.
\end{setup}

\subsection{Birational geometry of the universal family} \label{subsectionbirational}

The following basic statement is a birational variant of the situation in \cite{LTZZ10},
the proof follows their strategy.

\begin{lemma} \label{lemmabirationalLTZZ}
Let $X$ be a projective manifold such that $-K_X$ is nef. Let $\holom{\psi}{Z}{X}$
be a generically finite morphism from a projective manifold $Z$, and let $\holom{\varphi'}{Z}{Y'}$
be a fibration onto a projective manifold $Y'$. 
Suppose that we have
$$
K_{Z/Y'} \sim_\Q \psi^\star K_X + E_1 - E_2
$$
where $E_1$ and $E_2$ are effective $\Q$-divisors such that $\supp(E_1+E_2)$
does not surject onto $Y'$. Then the following holds:
 
If $\supp(E_1)$ is $\psi$-exceptional, then $\supp(E_2)$ is $\psi$-exceptional.
\end{lemma}

\begin{proof} Fix a very ample divisor $H$ on $X$.
We argue by contradiction and suppose that 
$E_2$ is not $\psi$-exceptional. 

We claim that $E_1-E_2$ is pseudoeffective. 
Assuming this for the time being let us see how to conclude: and let $C \subset Z$ be a curve
that is a complete intersection of pull-backs of general elements of $|H|$. 
Since $\supp(E_1)$ is $\psi$-exceptional, we have 
$E_1 \cdot C = 0$. Since the support of $- E_2$ is not $\psi$-exceptional and $C$ is cut out
by pull-backs of ample divisors, we have $(E_1-E_2) \cdot C = - E_2 \cdot C < 0$, a contradiction.

{\em Proof of the claim:} It is sufficient to show that
$$
K_{Z/Y'}  + \psi^\star (-K_X+\delta H) \sim_\Q  E_1-E_2 + \psi^\star (\delta H)
$$
is pseudoeffective for all $\delta>0$. Since $\supp(E_1+E_2)$
does not surject onto $Y'$,
we know that $K_{Z}$ and $\psi^\star(K_X)$ coincide along the general $\varphi'$-fibre. 
Thus, since $\psi$ is generically finite, the $\Q$-divisor $K_{Z/Y'}  + \psi^\star (-K_X+\delta H)$  is big on the general $\varphi'$-fibre. In particular for a sufficiently divisible $m \gg 0$ the higher direct images 
$$
\varphi'_\star (\sO_Z(m (K_{Z/Y'}  + \psi^\star (-K_X+\delta H))))
$$
are not zero sheaves. Moreover since $H$ is ample and $\delta>0$ the divisor
$\psi^\star (-K_X+\delta H)$ is semiample. Thus we have
$$
\psi^\star (-K_X+\delta H) \sim_\Q B_\delta
$$
where $B_\delta$ is an effective $\Q$-divisor such that the pair $(Z, B_\delta)$ is klt.
By \cite[Thm.4.13]{Cam04} the direct images $\varphi'_\star (\sO_{Z}(m (K_{Z/Y'}  + B_\delta)))$
are weakly positive. Since the relative evaluation map 
$$
(\varphi')^\star \varphi'_\star (\sO_{Z}(m (K_{Z/Y'}  + B_\delta))) 
\rightarrow 
\sO_{Z}(m (K_{Z/Y'}  + B_\delta))
$$
is generically surjective, this shows that $m (K_{Z/Y'}  + B_\delta)$ is weakly positive,
hence pseudoeffective.
\end{proof}

\begin{proposition} \label{propositionfundamentals}
In the situation of Setup \ref{setup}, the following holds:
\begin{enumerate}[1)]
\item We have $\supp(N_Y) \subset P$.
\item Every irreducible component of $\pi(\fibre{\varphi}{Y \setminus Y_0})$ has codimension at least two. In particular every $\varphi$-exceptional divisor is $\pi$-exceptional.
\item Let $B \subset Y_0$ be a prime divisor, and let $\varphi^\star B = \sum_{j=1}^l m_j B_j$ be the decomposition in prime divisors. If $m_j>1$ then $B_j$ is $\pi$-exceptional.
\end{enumerate}
\end{proposition}

\begin{proof} Denote by $(E_j)_{j=1, \ldots, l}$ the irreducible components of the $\pi$-exceptional locus $E$. The map $\pi$ is birational and $X$ is smooth, so we have
$$
K_\Gamma \sim \pi^\star K_X + \sum_{j=1}^l a_j E_j
$$
where $a_j>0$ for all $j$. Since the MRC fibration is almost holomorphic, we know that
$E=\supp(\sum_{j=1}^l a_j E_j)$ does not surject onto $Y$. 

{\em Proof of 1)} We have
\begin{equation} \label{birational}
K_{\Gamma/Y} \sim_\Q \pi^\star K_X + \sum_{j=1}^l a_j E_j - \varphi^\star N_Y,
\end{equation}
so by Lemma \ref{lemmabirationalLTZZ} the divisor $\varphi^\star N_Y$ is $\pi$-exceptional.
By the definition of $Y_0$ this implies that $\supp(N_Y) \subset P$.

{\em Proof of 2)} Assume that there exists a prime divisor $D \subset \Gamma$ that is not 
$\pi$-exceptional and such that $\varphi(D) \subset Y \setminus Y_0$.
By the definition of $Y_0$ this implies that $\varphi(D)$ has codimension at least two.
Denote by $\holom{\tau}{Y'}{Y}$ the composition of the blowup of $\varphi(D)$ with a resolution of singularities. Since $Y$ is smooth, we have
$$
K_{Y'} \sim \tau^\star K_Y + E_{Y'} \sim_\Q \tau^\star N_Y + E_{Y'}
$$
with $E_{Y'}$ an effective divisor whose support is equal to the $\tau$-exceptional locus.
In particular $\tau(\supp E_{Y'})$ contains $\varphi(D)$. 

Let now $\holom{\tau'}{Z}{\Gamma}$ be a resolution of indeterminacies of $\Gamma \dashrightarrow Y'$ such that $Z$ is smooth, and denote by
$\holom{\varphi'}{Z}{Y'}$ the induced fibration. 
Let $D' \subset Z$
be the strict transform of $D$. By construction we have $\varphi'(D') \subset 
\supp(\tau^\star N_Y + E_{Y'})$. Since $\tau^\star N_Y + E_{Y'}$ is a canonical divisor,
we can argue as in the proof of 1) to see that the support of $(\varphi')^\star (\tau^\star N_Y + E_{Y'})$ is $\pi \circ \tau'$-exceptional. Thus  
$D'$ is $\pi \circ \tau'$-exceptional. Since $D'$ is not $\tau'$-exceptional, the divisor
$D=\tau'(D')$ is $\pi$-exceptional, a contradiction.

{\em Proof of 3)} By \cite[Prop.4.1.12]{Laz04a} 
we can choose a generically finite covering $\holom{\tau'}{Y'}{Y}$
by some projective manifold $Y'$ that ramifies with multiplicity $m_j$ over $B$ and over every
other prime divisor contained in the branch locus the general $\varphi$-fibre
is smooth. The fibre product $\Gamma \times_Y Y'$ is not normal along
a divisor mapping onto $B_j$, and we denote by $Z$ its normalisation. 
By \cite[Prop.2.3]{Rei94} we have
\begin{equation} \label{basechange}
K_{Z/Y'} = (\tau')^\star K_{\Gamma/Y} - A + C_1 - C_2
\end{equation}
where $\holom{\tau'}{Z}{\Gamma}$ is the natural map, the divisor $A$ is effective s.t.
$B_j \subset \tau'(A)$
and $C_1$ and $C_2$ are effective divisors taking into account that the base change formula
for relative canonical sheaves does not hold over the non-flat locus.
In particular $\tau'(\supp(C_1-C_2))$ is contained in $\varphi$-exceptional divisors, hence
$\supp(C_1-C_2)$ is $\pi \circ \tau'$-exceptional by 2). These properties do not change
if we replace $Z$ by a resolution of singularities, so we suppose without loss of generality
that $Z$ is smooth. 
Combining \eqref{basechange} and \eqref{birational} we obtain
$$
K_{Z/Y'} = (\tau')^\star \pi^\star K_X + \pi^\star (\sum a_j E_j) + C_1 - (\varphi^\star N_Y + A + C_2).
$$
By what precedes the support of $\pi^\star (\sum a_j E_j) + C_1$ is 
$\pi \circ \tau'$-exceptional, so the support of $\varphi^\star N_Y + A + C_2$
is $\pi \circ \tau'$-exceptional by Lemma \ref{lemmabirationalLTZZ}.
Since $\tau'(A)$ contains $B_j$, this implies that $B_j$ is $\pi$-exceptional.
\end{proof}

\begin{remark} \label{remarkholomorphicconstant}
The following elementary observation will be quite useful: any holomorphic
function $s: Y_0 \rightarrow \C$ is constant. Indeed the pull-back $\varphi^\star s$ is holomorphic on 
$(\Gamma \setminus E) \cap \fibre{\varphi}{Y_0}$ which we identify to a Zariski open subset of $X$. By 
Proposition \ref{propositionfundamentals}, 2) the complement of this set has codimension at least two,
so $\varphi^\star s$ extends to $X$. Thus it is constant.

Since the rank of a linear map is determined by the nonvanishing of its minors, this implies the following: let
$$
\sO_{Y_0}^{\oplus r_1} \rightarrow \sO_{Y_0}^{\oplus r_2}
$$
be a morphism between trivial vector bundles on $Y_0$. Then the rank of this map is constant.
\end{remark}

\subsection{Triviality of direct images} \label{sectiontriviality}

Following the terminology in Setup \ref{setup}, the aim of this subsection is to prove that, for $m$ large enough,
$$\sV_p := \varphi_\star (\sO_\Gamma(p L + p m E - \frac{p}{r} \varphi^\star \det \varphi_\star 
(\sO_\Gamma(L + m E))))$$ 
is a trivial vector bundle over $Y_0$ for every $p\in\mathbb{N}$ sufficiently divisible.

We first sketch the idea of the proof and describe the outline of the subsection. 
Our argument is very close to \cite{Cao16}. In fact, thanks to Proposition 
\ref{propositionfundamentals}, we know that $-K_{\Gamma /Y}$ can be written as the sum of a nef line bundle and a divisor supported in the exceptional locus of 
$\pi$. Then most arguments in \cite{Cao16} still work for our case cf. Lemma \ref{psflemma} and Lemma \ref{lmpost}. 

The main difference is that, because of the exceptional locus $E \cap \varphi^{-1} (Y_0)$, we can \`{a} priori only  get the flatness of $\sV_p$ 
over $Y_0 \setminus \varphi_\star (E)$.
To overcome this difficulty, we need a new observation: Lemma \ref{lemmastabilise}, namely, 
for $m$ large enough, there is an isomorphism over $Y_0$:
\begin{equation}\label{isoointr}
(\det\varphi_\star (\sO_\Gamma(L+ m E))) \otimes \sO_{Y_0} 
\simeq (\det\varphi_\star (\sO_\Gamma(L + (m+1) E))) \otimes \sO_{Y_0}.
\end{equation}
Together with Lemma \ref{psflemma} and Lemma \ref{lmpost}, we prove in Proposition \ref{flatness}
that $\sV_p$ is weakly positively curved over $Y_0$ and $\pi_\star \varphi^\star c_1 (\sV_p) =0$.
Finally, thanks to Corollary \ref{trivialexten} and the property of $Y_0$ (cf. Proposition \ref{propositionfundamentals}), we prove in Proposition
\ref{propositiontrivial} that $\sV_p$ is a trivial vector bundle on $Y_0$.

\medskip

Before proving the main proposition \ref{propositiontrivial} in this subsection, we need a series of technical lemmas. 
The first one comes from Viehweg's diagonal trick \cite[Thm 6.24]{Vie95} and Q. Zhang's argument \cite{Zha05}. 
We refer to \cite[Prop 3.12]{Cao16} \cite[Thm 3.13]{CP17} for similar argument.
\begin{lemma}\label{psflemma}
In the situation of Setup \ref{setup}, for every $m \in \N$, 
\begin{equation}\label{pshnew}
 A - \frac{1}{r}\pi_\star \varphi^\star c_1(\varphi_\star (\sO_\Gamma(L+ m E))) \qquad\text{is pseudo-effective on } X . 
\end{equation}
\end{lemma}

\begin{proof}
Let $E'\subset \Gamma$ be the union of non reduced and non flat locus of  $\varphi$. 
Let $\Gamma^{(r)}$ be a desingularisation of the $r$-times $\varphi$-fibrewise product $\Gamma \times_Y \cdots \times_Y \Gamma$, and let $\pr_i : \Gamma^{(r)} \rightarrow \Gamma$ be the $i$-th directional projection. Let $\varphi_r : \Gamma^{(r)} \rightarrow Y$ be the natural fibration.  
Set $L^{(r)} := \sum\limits_{i=1}^r \pr_i ^\star L$ and $E^{(r)} :=\sum\limits_{i=1}^r \pr_i ^\star  E$.
We have a natural morphism
\begin{equation}\label{natmor}
\det \varphi_\star (\sO_\Gamma(L +mE)) \rightarrow (\varphi_r)_\star (\sO_{\Gamma^{(r)}}(L^{(r)} + m E^{(r)} +\Delta))
\end{equation}
for some divisor $\Delta$ supported in $\sum\limits_{i=1}^r \pr_i ^\star (E')$.
Set 
$$
L' :=L^{(r)} + m E^{(r)} +\Delta - \varphi_r ^\star c_1(\varphi_\star (\sO_\Gamma(L +mE))).
$$
As the restriction of \eqref{natmor} on a generic point $y\in Y$ is non zero, \eqref{natmor} induces a nontrivial section in 
$H^0 (\Gamma^{(r)} , L')$. Therefore $L'$ is an effective line bundle on $\Gamma^{(r)}$. 

\medskip

We now compare $K_{\Gamma^{(r)}/Y}$ with $\sum\limits_{i=1}^r \pr_i ^\star (\pi^\star K_X)$. Let $\varphi^{(r)} : \Gamma^{(r)}\rightarrow Y$ be the natural morphism.
Then 
$$K_{\Gamma^{(r)}/Y}  - \sum\limits_{i=1}^r \pr_i ^\star (\pi^\star K_X) = -(\varphi^{(r)})^\star K_Y  + (K_{\Gamma^{(r)}} - \sum\limits_{i=1}^r \pr_i ^\star K_\Gamma) + 
\sum\limits_{i=1}^r \pr_i ^\star (K_\Gamma- \pi^\star K_X)$$
By Proposition \ref{propositionfundamentals}, $(\varphi^{(r)})^\star K_Y$ is equivalent to a divisor supported in $\sum\limits_{i=1}^r \pr_i ^\star E$.
By construction, $K_{\Gamma^{(r)}} - \sum\limits_{i=1}^r \pr_i ^\star K_\Gamma$ is supported in $\sum\limits_{i=1}^r \pr_i ^\star E'$, and $\sum\limits_{i=1}^r \pr_i ^\star (K_\Gamma- \pi^\star K_X)$
is supported in $\sum\limits_{i=1}^r \pr_i ^\star E$.
As a consequence,  we can find a divisor (not necessary effective) $\Delta'$ supported in $\sum\limits_{i=1}^r \pr_i ^\star (E' +E)$ such that
$$- K_{\Gamma^{(r)}/Y} +\Delta' = - \sum\limits_{i=1}^r \pr_i ^\star (\pi^\star K_X) .$$ 
In particular, $- K_{\Gamma^{(r)}/Y} +\Delta' $ is nef. 

\medskip

By using \cite[Prop 2.10]{Cao16}, there exists an ample line bundle $A_Y$ on $Y$ such that for every $p\in\mathbb{N}$ and for $q$ large enough (with respect to $p$),
$$
H^0 (\Gamma^{(r)}, pq K_{\Gamma^{(r)} /Y} + pq (- K_{\Gamma^{(r)}/Y} +\Delta') + p L' +\varphi_r ^\star A_Y )$$
$$
\rightarrow 
H^0 (\Gamma^{(r)} _y , pq K_{\Gamma^{(r)} /Y} + pq (- K_{\Gamma^{(r)}/Y} +\Delta') + p L' +\varphi_r ^\star A_Y ) 
$$
is surjective for a generic $y \in Y$.

By restricting the above morphism on diagonal\footnote{It is only well defined on $\Gamma\setminus E'$, but by adding some divisor supported in $E'$, 
we can extend it on $\Gamma$.}, we can find some effective divisor $F_{p,q}$ (depending on $p$ and $q$) supported in $E \cup  E'$ such that
$$
H^0 (\Gamma,   F_{p,q} + p r L  -  p \varphi^\star c_1(\varphi_\star (\sO_\Gamma(L+mE))) +\varphi^\star A_Y )
$$
$$
\rightarrow 
H^0 (\Gamma_y ,   F_{p,q} + p rL  - p \varphi^\star c_1(\varphi_\star (\sO_\Gamma(L+mE))) +\varphi^\star A_Y ) 
$$
is surjective.
In particular, $F_{p,q} + p r L  - p \varphi^\star c_1(\varphi_\star (\sO_\Gamma(L+mE))) +\varphi^\star A_Y$ is effective.
Thanks to Proposition \ref{propositionfundamentals}, $\pi_\star (F_{p,q}) \subset \pi_\star (E \cup E')$ has codimension at least two. 
Therefore 
$$
\pi_\star 
\left(
c_1
\left(
r L  - \varphi^\star c_1(\varphi_\star (\sO_\Gamma(L+mE))) +\frac{1}{p}\varphi^\star A_Y)
\right)
\right)
$$ 
is pseudoeffective on $X$.
The lemma is proved by letting $p \rightarrow +\infty$.
\end{proof}

The second lemma is very close in spirit to Lemma \ref{lemmabirationalLTZZ}, namely, the proof uses 
the positivity of direct images as well as Q. Zhang's technique \cite{Zha05} in the study of manifolds with nef anticanonical bundles.

\begin{lemma}\label{lmpost}
Let $E_1$ be an effective $\mathbb{Q}$-divisor satisfying $\varphi_\star (E_1) \subsetneq Y$ and let $L_Y$ be a $\mathbb{Q}$-line bundle on $Y$.
Let $a >0$ be some constant such that the line bundle
$$L_1  := a L + E_1 + \varphi^\star L_Y $$ 
is pseudo-effective. Then for $c\in \mathbb{N}$ large enough, $\varphi_\star (L_1 + c E)$ is weakly positive curved (cf. Definition \ref{weakpocu}) over $Y_1$, where $Y_1$ is the locally free locus of $\varphi_\star (L_1 + c E)$.
In particular, $\det \varphi_\star (L_1 + c E)$ is pseudoeffective on $Y$. 
\end{lemma}

\begin{proof}
By construction, we can find some effective divisor $E_2$ supported in $E$ such that $-K_{\Gamma} + E_2 = -\pi^\star K_X$. 
Let $A_Y$ be an ample line bundle on $Y$ such that $A_Y + L_Y$ is still ample.

For every $\epsilon \in\bQ^+$, 
$$-K_\Gamma +E_2 +\ep a L + \ep \varphi^\star ( L_Y +A_Y) =\pi^\star (-K_X + \ep a A) + \ep \varphi^\star ( L_Y +A_Y)$$ 
is semi-ample. 
Note that $\ep E_1 + \varphi^\star K_Y \sim_\bQ D$ for some $\mathbb{Q}$-effective divisor $D$ satisfying 
$\varphi_\star ( D) \subsetneq Y$. Then the $\mathbb{Q}$-line bundle
$$-K_{\Gamma /Y} + E_2 + \epsilon L_1 = (-K_\Gamma +E_2 +\ep a L + \ep \varphi^\star ( L_Y +A_Y))  + (\ep E_1 +  \varphi^\star K_Y) -\ep \varphi^\star A_Y.$$
can be equipped with a metric $h_{1,\ep}$ such that 
$$i\Theta_{h_{1,\ep}} (-K_{\Gamma /Y} + E_2 + \epsilon L_1) \geq -\ep\varphi^\star\omega_Y$$
and $h_{1,\ep} |_{\Gamma_y}$ is smooth for a generic fiber $\Gamma_y$, where $\omega_Y$ is a K\"{a}hler metric representing $A_Y$.

Since $L_1$ is pseudo-effective, there exists a possibly singular metric $h$ on $L_1$ such that $i\Theta_{h} (L_1) \geq 0$. Then we can find a $m$ (depends only on $h$)
such that $\mathcal{I} (h ^{\frac{1}{m}} |_{\Gamma_y}) = \mathcal{O}_{\Gamma_y}$ for a generic fiber $\Gamma_y$.
Therefore, for every $0< \ep < \frac{1}{m}$, the metric $h_\epsilon := m h_{1,\ep}  + (1-m \ep) h$ defines a metric on  
$$-m K_{\Gamma /Y}  + m E_2 + L_1 = m (-K_{\Gamma /Y} + E_2 + \epsilon L_1 ) +(1-m\epsilon) L_1$$
with $i\Theta_{h_\ep} (-m K_{\Gamma /Y}  + m E_2 + L_1)\geq -m\epsilon \varphi^\star\omega_Y$ on $\Gamma$ and 
$\mathcal{I} (h_\ep ^{\frac{1}{m}} |_{\Gamma_y}) = \mathcal{O}_{\Gamma_y}$.

By applying Proposition \ref{lowercurcontrol} to the line bundle $(-m K_{\Gamma /Y}  + m E_2 + L_1, h_\ep)$,  we can find a metric $\widetilde{h}_\epsilon$ on 
$$
\varphi_\star (\sO_\Gamma(m E_2 +  L_1)) = \varphi_\star (\sO_\Gamma(m K_{\Gamma/Y} + (-m K_{\Gamma /Y}  + m E_2 + L_1)))
$$
such that 
\begin{equation}\label{semipocur}
i\Theta_{\widetilde{h}_\epsilon} (\varphi_\star (\sO_\Gamma(m E_2 +  L_1))) \succeq -m\epsilon \omega_Y  \otimes \Id_{\End (\varphi_\star (\sO_\Gamma(m E_2 +  L_1)))}\qquad\text{on } Y_{1,m} ,
\end{equation}
where $Y_{1,m} $ is the locally free locus of $\varphi_\star (\sO_\Gamma(m E_2 +  L_1))$.
\medskip

As $m$ depends only on $h$, we can find a constant $c$ independent of $\ep$ such that $c E \geq m E_2$. Then \eqref{semipocur} implies the existence of possible singular metrics $\widetilde{h}_{c,\epsilon}$ such that $i\Theta_{\widetilde{h}_{c,\epsilon}} (\varphi_\star (\sO_\Gamma(c E +  L_1))) \succeq -m\epsilon \omega_Y \otimes \Id_{\End (\varphi_\star (\sO_\Gamma(c E +  L_1)))}$ on $Y_1$. 
The first part of the lemma is thus proved.

The second part of lemma comes from the similar argument in \cite[Cor 2.26, first part]{Pau16}.
\end{proof}

Before proving the main proposition in this subsection, we need two more lemmas.
The first lemma is to compare $\varphi_\star \sO_\Gamma(L+ m E)$ with $\varphi_\star \sO_\Gamma(L + (m+1) E)$ over $Y_0$:
\begin{lemma}\label{lemmastabilise}
In the situation of Setup \ref{setup}, there exists a $m_0 \in \N$ such that for all $m \geq m_0$
the natural inclusions
\begin{equation}\label{incl}
\det \varphi_\star (\sO_\Gamma(L+ m E)) \rightarrow \det\varphi_\star (\sO_\Gamma(L + (m+1) E)) 
\end{equation}
induce an isomorphism
\begin{equation}\label{isoo}
(\det\varphi_\star (\sO_\Gamma(L+ m E))) \otimes \sO_{Y_0} 
\simeq (\det\varphi_\star (\sO_\Gamma(L + (m+1) E))) \otimes \sO_{Y_0}.
\end{equation}
\end{lemma}

\begin{proof}
Note that, the effective divisor 
$$
c_1(\varphi_\star (\sO_\Gamma(L + (m+1) E))) - c_1(\varphi_\star (\sO_\Gamma(L+ m E)))
$$
is supported in $\varphi_\star (E)$,  which is a fixed strict subvariety of $Y$ independent of $m$.

If  we assume by contradiction that \eqref{isoo} is not an isomorphism on $Y_0$ for infinitely many $m$, we can find an effective divisor $\widetilde{D}$ in $Y$
such that $\widetilde{D}$ is not contained in $Y\setminus Y_0$ and a sequence $a_m \rightarrow +\infty$ such that 
$$
c_1(\varphi_ \star (\sO_\Gamma(L+ m E))) -
c_1(\varphi_ \star (\sO_\Gamma(L))) \geq a_m \cdot [\widetilde{D}] .
$$
Combining this with Lemma \ref{psflemma}, we know that 
$$
A- \frac{1}{r}\pi_\star \varphi^\star c_1(\varphi_ \star (\sO_\Gamma(L))) -  
\frac{a_m}{r}\pi_\star \varphi^\star [\widetilde{D}] 
$$
is a pseudoeffective class on $X$.
By the definition of $Y_0$ and $\widetilde{D}$, the class 
$\pi_\star \varphi^\star [\widetilde{D}]$ is a non trivial effective class on $X$. 
We get thus a contradiction as $a_m \rightarrow +\infty$.
\end{proof}

The following lemma tells us that, for every $c >0$, 
$L+ c E - \frac{1}{r}\varphi^\star c_1(\varphi_ \star (\sO_\Gamma(L+ m E)))$ has no positivity in the horizontal direction. 
More precisely, we have
\begin{lemma}\label{psflemm}
In the situation of Setup \ref{setup}, let $m_0 \in \mathbb{N}$ as in Lemma \ref{lemmastabilise}, and let $m \geq m_0$ be an integer. 

Let $\theta$ be a closed  $(1,1)$-current on $Y$ such that $\theta \cdot \mathcal{C} >0$ 
for some movable compact curve $\mathcal{C}$ contained in $Y_0$.
Then the class
$$
L+ c E - \frac{1}{r}\varphi^\star c_1(\varphi_ \star (\sO_\Gamma(L+ m E))) - \varphi^\star \theta
$$ 
is not pseudoeffective for any constant $c >0$. 
\end{lemma}

\begin{remark*}
Note that we {\em do not} assume that $\theta$ is a positive current.
\end{remark*}

\begin{proof}
We assume by contradiction that there exists some $c \geq m$ such that 
$$
T:= L+ c E - \frac{1}{r}\varphi^\star c_1(\varphi_ \star (\sO_\Gamma(L+ m E))) - \varphi^\star \theta
$$ 
is pseudoeffective on $\Gamma$. We can thus apply  Lemma \ref{lmpost} to $T$ and  find a $c_1 >0$ such that 
$$
c_1(\varphi_\star (\sO_\Gamma(T + c_1 E))) \qquad\text{ is pseudoeffective on } Y.
$$
Combining this with the definition of $T$, we have
\begin{equation}\label{lemlater}
\det \varphi_\star (\sO_\Gamma(L+ c E + c_1E)) \geq \det \varphi_ \star (\sO_\Gamma(L+ m E)) + r \theta   \qquad\text{on } Y.
\end{equation}

Thanks to Lemma  \ref{lemmastabilise},   the effective class 
$$c_1(\varphi_\star (\sO_\Gamma(L+ c E + c_1 E))) - 
c_1(\varphi_ \star (\sO_\Gamma(L+ m E)))$$ 
is supported in $Y\setminus Y_0$. As $\mathcal{C}$ 
is contained in $Y_0$,
we have 
$$
c_1(\varphi_\star (\sO_\Gamma(L+ c E + c_1 E)))  \cdot \mathcal{C} = 
c_1(\varphi_\star (\sO_\Gamma(L+ m E)))  \cdot \mathcal{C} .
$$
We get thus a contradiction to \eqref{lemlater} since $\theta \cdot \mathcal{C} >0$ and $\mathcal{C}$ is movable by assumption.
\end{proof}

By combining Lemma  \ref{psflemma}, \ref{lmpost} and \ref{psflemm}, we can prove the following key proposition.

\begin{proposition}\label{flatness}
In the situation of Setup \ref{setup}, let $m_0 \in \mathbb{N}$ as in Lemma \ref{lemmastabilise}, and let $m \geq m_0$ be an integer. 
Fix a K\"ahler metric $\omega_X$ on $X$ and let $\beta$ be a smooth form representing 
$\frac{1}{r} \det \varphi_\star (\sO_\Gamma( L + m E))$ on $Y$.

{\em (i):} For every $p\in\bN$, let $Y_{1,p}$ be the locally free locus of $\varphi_\star (\sO_\Gamma(p L + p m E))$.
Then for every $\ep >0$, we can find a metric $h_{\ep, p}$
on $ \varphi_\star (\sO_\Gamma(p L + p m E)) $ over $Y_0 \cap Y_{1,p}$ 
such that
$$
i\Theta_{h_{\ep , p}} (\varphi_\star (\sO_\Gamma( p L + p m E)) ) 
\succcurlyeq (-\ep \omega_X +p \beta)\otimes \Id_{\End (\varphi_\star (\sO_\Gamma(p L + p m E)) )} \text{ on } Y_0 \cap Y_{1,p}.
$$

{\em (ii):} For every $p\in\mathbb{N}$ divisible by $r$, we have
$$
\frac{1}{r_p}\pi_\star \varphi^\star (c_1 (\varphi_\star \sO_\Gamma(p L + pm E))) =
\frac{p}{r} \pi_\star \varphi^\star ( c_1 (\varphi_\star \sO_\Gamma( L + m E))) 
$$
in $H^{1,1} (X, \mathbb{R})$, where $r_p$ is the rank of $\varphi_\star (\sO_\Gamma(p L + pm E))$.
\end{proposition}

\begin{proof}[Proof of {\em (i):}]
Thanks to Lemma \ref{psflemma}, we can find a constant $c >0$ such that $L+c E - \varphi^\star \beta$ is pseudoeffective on $\Gamma$.
It means that we can find a possible singular metric $h_L$ on $L +cE$ such that $i\Theta_{h_L} (L+cE) \geq \varphi^\star \beta$.
As a consequence, by applying Lemma \ref{lmpost} to $L+cE -\varphi^\star \beta$, there exists a constant $c' \geq m$
such that, for every $\epsilon >0$, we can find a possible singular hermitian metric
$h_\epsilon$ on $\varphi_\star (L +c' E)$ such that 
\begin{equation}\label{semp}
i\Theta_{h_\ep} \varphi_\star (L + c' E) \succcurlyeq (-\ep \omega_X + \beta) \otimes \Id_{\End (\varphi_\star (L + c' E) )} \qquad\text{on }Y_{1,1, c'} ,
\end{equation}
where $Y_{1,1, c'}$ is the locally free locus of $\varphi_\star (L + c' E)$.
Thanks to \eqref{isoo}, we have 
$$
\varphi_\star (\sO_\Gamma(L + c' E)) \otimes \sO_{Y_0 }
= \varphi_\star (\sO_\Gamma(L + m E )) \otimes \sO_{Y_0}.
$$
Together with \eqref{semp}, for every $\ep >0$, we can find a metric $h_{\tilde \ep}$
on $ \varphi_\star (\sO_\Gamma(L + m E)) $ such that
\begin{equation}\label{sempos}
i\Theta_{h_{\tilde \ep}} (\varphi_\star (\sO_\Gamma(L + m E)) ) \succcurlyeq (-\ep \omega_X +\beta)\otimes \Id_{\End (\varphi_\star (\sO_\Gamma(L + m E)) )} \text{ on }Y_{1,1} \cap Y_0  .
\end{equation}

\medskip

We consider the natural morphism
\begin{equation}\label{namor}
\mbox{Sym}^{p} \varphi_\star \sO_\Gamma(L + m E) \otimes \sO_{Y_0} \rightarrow  \varphi_\star \sO_\Gamma(p L + p m E)\otimes \sO_{Y_0}. 
\end{equation}
By \eqref{surgene}, it is surjective on a generic point. By applying Proposition \ref{quotientweakpositive} to \eqref{namor},
\eqref{sempos} induces thus a possible singular metric $h_{\ep, p}$ on $\varphi_\star (\sO_\Gamma(pL + pm E)) $ such that
\begin{equation}\label{sempospcase}
i\Theta_{h_\ep , p} (\varphi_\star (\sO_\Gamma( pL + p m E)) ) \succcurlyeq (-\ep \omega_X + p\beta) \otimes \Id_{\End (\varphi_\star \sO_\Gamma( pL + p m E))} \text{ on }Y_{1,p} \cap Y_0  .
\end{equation}
The part $(i)$ is thus proved.
\end{proof}

For every $p\in \bN$ divisible by $r$, set
$$\sE_p := (\mbox{Sym}^{p} \varphi_\star \sO_\Gamma(L + m E))\otimes (\det \varphi_\star \sO_\Gamma(L + m E))^{\otimes \frac{-p}{r} } $$
and
\begin{equation}\label{direc}
\sV_p := \varphi_\star \sO_\Gamma(p L + pm E)) \otimes (\det \varphi_\star \sO_\Gamma(L + m E))^{\otimes \frac{-p}{r} }  .
\end{equation}

Before proving the second statement, we observe the following: 
let $\tau: \tilde C \rightarrow Y_0 \cap Y_{1,1}$ be a morphism from a smooth compact curve 
to $Y_0$ that is very general (i.e. the image is not contained in the locus
where the metrics constructed above are singular). Then, for every $p\in \bN$ divisible by $r$,
$\tau^\star \sE_p$ is numerically flat. Indeed its first Chern class is zero by construction
and $\tau^\star \sE_p$ is nef by $(i)$.

\begin{proof}[Proof of {\em (ii):}]
Let $\alpha$ be a smooth form representing  
$
\frac{1}{r_p}
c_1 (\varphi_\star (\sO_\Gamma(p L + pm E))).
$
Let $C \subset X$ be a very general complete intersection curve cut out by sections
of a very ample line bundle $A_X$ on $X$. Since $C \cap Z = \emptyset$, we identify  with its strict transform in $\Gamma$.
Since $\varphi_\star \sO_\Gamma(L + m E)$ is reflexive on $Y_0$, we know that $Y_0 \setminus Y_{1,1}$ 
has codimension at least two in $Y_0$. Thus it follows by Proposition \ref{propositionfundamentals} that $\pi(\fibre{\varphi}{Y \setminus Y_{1,1}})$ 
has codimension at least two. Thus we can assume without loss of generality that
the movable curve $\varphi(C)$ is contained in $Y_{1,1} \cap Y_0$.
Let $\tau: \widetilde{C} \rightarrow \varphi(C)$ be the normalisation.

Since $N_1(X)$ is generated by cohomology classes of the form $A_X^{\dim X-1}$
with $A_X$ a very ample line bundle on $X$, it is sufficient, by duality and the projection formula, to prove that 
\begin{equation}\label{eqli}
\int_{\widetilde{C}} \tau^\star (\alpha-p \beta)=0 . 
\end{equation}
The map \eqref{namor} induces a generically surjective morphism
$$
\tau^\star \sE_p \rightarrow \tau^\star  \sV_p\qquad\text{on } \widetilde{C}. 
$$
We observed above that $\tau^\star  \sE_p$ is numerically flat on $\widetilde{C}$, hence also $\tau^\star  (\sV_p)$ \cite[Ex.6.4.17]{Laz04b}.
Yet by definition
$c_1(\sV_p)$ is represented by $r_p (\alpha - p \beta)$. 
Thus we get
\begin{equation}\label{sepos}
\int_{\widetilde{C}}\tau^\star (\alpha - p \beta) \geq 0 . 
\end{equation}
In the other direction, as $p L$ is also the pull-back of some very ample bundle on $X$,
we can replace $L$ by $pL$ in Lemma \ref{psflemma}. Therefore
$p L + c_3 E - \varphi^\star  \alpha$ 
is pseudoeffective on $\Gamma$ for some $c_3$ large enough.
Hence
$$
L + \frac{c_3}{p} E - \frac{\varphi^\star  \alpha}{p} = L + \frac{c_3}{p} E - \varphi^\star  \beta - \varphi^\star  (\frac{\alpha-p \beta}{p})
$$ 
is pseudoeffective. By applying Lemma \ref{psflemm}, we get
$$
\int_{\widetilde{C}}\tau^\star (\alpha - p \beta) \leq 0. 
$$
Together with \eqref{sepos} this proves \eqref{eqli}.
\end{proof}

Now we can prove the main result in this subsection.

\begin{proposition} \label{propositiontrivial}
In the situation of Setup \ref{setup}, let $m_0 \in \mathbb{N}$ as in Lemma \ref{lemmastabilise}, and let $m \geq m_0$ be an integer. 
Fix a sufficiently divisible $p\in \mathbb{N}$ and let $\sV_p$ be the direct image defined in \eqref{direc}.
Then $\sV_p$ is a trivial vector bundle on $Y_0$. 
\end{proposition}

\begin{proof}
Denote by $Y_{1,p} \subset Y$ the locally free locus of $\sV_p$. Since $\sV_p$ is reflexive on $Y_0$ we know that
the complement of $Y_0 \cap Y_{1,p}$ in $Y_0$ has codimension at least two. By Proposition \ref{facts}, 3)
it is sufficient to show that $\sV_p \otimes \sO_{Y_0 \cap Y_{1,p}}$ is a trivial vector bundle.

The image of the $\pi$-exceptional locus $Z := \pi (E)$ has codimension at least two.
Denote by $j: X \setminus Z \rightarrow X$ the inclusion.

Let $A_X$ be a very ample line bundle on $X$, and let
$$
S = H_1 \cap \cdots \cap H_{n-2}
$$
be a surface cut out by general elements $H_1 , \cdots , H_{n-2} \in |A_X|$. 

Then $Z_S := Z \cap S$ is a finite set and $S\setminus Z_S$ is simply connected, 
since by the Lefschetz theorem $\pi_1(S) \simeq \pi_1(X)$ is trivial.
Let $\pi^{-1} (S\setminus Z_S)$ be the strict transform of $S\setminus Z_S$.
By using Proposition \ref{propositionfundamentals}, we see that
$\pi^{-1} (S\setminus Z_S) \setminus \varphi^{-1} (Y_0)$ is also finite.

The vector bundle $\varphi^\star\sV_p |_{\pi^{-1}  (S\setminus Z_S)}$ is weakly positively curved
by Proposition \ref{flatness}(i) and $c_1 (\varphi^\star\sV_p |_{\pi^{-1}  (S\setminus Z_S)})=0$ Proposition \ref{flatness}(ii).
Therefore, by Proposition \ref{trivialexten} its extension to $S$ is numerically flat.
Since $\pi_1 (S)=1$ this implies that this extension, hence  $\varphi^\star\sV_p |_{\pi^{-1}  (S\setminus Z_S)}$ is  a trivial vector bundle.

\medskip

Since $Z$ has codimension at least two, we have isomorphisms
\begin{equation}\label{isobir}
H^0 (\Gamma\setminus E,  \varphi^\star \sV_p) 
\simeq H^0(X \setminus Z, \varphi^\star \sV_p) 
\simeq
H^0(X, (j_\star (\varphi^\star \sV_p))^{\star \star }). 
\end{equation}
Consider now a curve 
$$
C = S \cap H_{n-1}
$$
where $H_{n-1} \in |A_X|$ is general. 
Up to replacing $A_X$ by some positive multiple, we can suppose that
we have a surjective morphism
$$
H^0 (\Gamma\setminus E,  \varphi^\star \sV_p) 
 \rightarrow H^0 (C, \varphi^\star  \sV_p) .
$$ 
As $Z \cap S = \emptyset$ we know that $(\varphi^\star  \sV_p)|_C$ 
is a trivial vector bundle.
Thus the surjectivity above implies the existence of section
$$
s_1, s_2 ,\cdots s_{r_p} \in H^0 (\Gamma\setminus E,  \varphi^\star \sV_p)
$$ 
such that 
$$
s_1 \wedge \cdots \wedge s_{r_p} \in H^0 (\Gamma\setminus E,  \det \varphi^\star \sV_p) \text{ is a non trivial section}.
$$
Thanks to \eqref{isobir}, it induces a section of $H^0 (X, (j_\star  (\det \varphi^\star \sV_p))^{\star\star })$.
Yet $X$ is compact and $c_1 (j_\star  (\det \varphi^\star \sV_p)) =0$ by using Proposition \ref{flatness}. 
So the nonvanishing section is a non-zero constant.
In particular $s_1 \wedge \cdots \wedge s_{r_p}$ is a non-zero constant.

\medskip

We claim that $s_i|_{\fibre{\varphi}{Y_0} \setminus E} = \varphi^\star  \tau_i$ for some 
$\tau_i \in H^0(Y_0 , \sV_p)$. If the claim is proved, since
$$
s_1 \wedge \cdots \wedge s_{r_p} = \varphi^\star  (\tau_1 \wedge \cdots \wedge \tau_{r_p})
$$
is a non-zero constant, these sections define the trivialisation of $\sV_p$. The proposition is proved.

{\em Proof of the claim.} Since $\Gamma \setminus E \rightarrow Y$ is not proper, this is not
totally obvious. However, since $\varphi$ has connected fibers, 
the sections $s_i \in H^0(\Gamma \setminus E, \varphi^\star \sV_p)$ 
induce sections $\tilde \tau_i \in H^0(Y_0 \setminus \varphi(E), \sV_p)$. 

We need to prove that $\tilde \tau_i$
extends to a section $\tau_i \in H^0(Y_0, \sV_p)$: since
$\sV_p \otimes \sO_{Y_0}$ is reflexive, it is enough to show that $\tilde \tau_i$
extends in the general point of every divisor $B \subset Y$.
By definition 
of $Y_0$ the morphism $\Gamma \setminus E \rightarrow Y$ dominates
every divisor $B \subset Y_0$. Moreover, by Proposition \ref{propositionfundamentals},3)
the general fibre over $B$ has at least one reduced component. Thus there exists
a local section of $\Gamma \setminus E \rightarrow Y$ in a general point of $B$.
Using this section we see that $\tilde \tau_i$ is bounded in a general point of $B$, hence it extends by Riemann's extension theorem.
\end{proof}

\subsection{Proof of Theorem \ref{theoremmain}}
\label{subsectionproof}

In the situation of Setup \ref{setup} we use the notation of Proposition \ref{flatness}. 
Moreover we choose $m \geq m_0$ as
in Proposition \ref{propositiontrivial}. For a $p_0 \in \N$ sufficiently divisible we know by Proposition \ref{propositiontrivial}
that the vector bundle $\sV_{p_0} \otimes \sO_{Y_0}$ is trivial. Thus we have
$$
\varphi_\star (\sO_\Gamma(p_0(L+mE))) \otimes \sO_{Y_0} \simeq ((\det \varphi_\star (\sO_\Gamma(L+mE)))^{\otimes \frac{p_0}{r}})^{\oplus r_{p_0}} \otimes \sO_{Y_0},
$$
In particular $\det \varphi_\star (\sO_\Gamma(p_0(L+mE))) \otimes \sO_{Y_0}$ is divisible by $r_{p_0}$ in the Picard group. In order to simplify the notation we replace $L$ by $p_0L$ and $mE$ by $p_0 mE$, hence, up to this change of notation, 
the line bundle $\det \varphi_\star (\sO_\Gamma(L+mE)) \otimes \sO_{Y_0}$ is divisible by $r$ and
Proposition \ref{propositiontrivial} holds for all $p \in \N$.

By \eqref{surgene} the natural map
\begin{equation} \label{symmap}
\mbox{Sym}^p (\varphi_\star (\sO_\Gamma(L+mE))) \otimes \sO_{Y_0} \rightarrow \varphi_\star (\sO_\Gamma(p(L+mE))) \otimes \sO_{Y_0}
\end{equation}
is generically surjective. Using the notation of Proposition \ref{flatness}
the twist of this map by 
$(\det \varphi_\star (\sO_\Gamma(L + m E)))^{\otimes \frac{-p}{r}}$
is the map
$$
\mbox{Sym}^p \sV_1  \otimes \sO_{Y_0} 
\rightarrow \sV_p  \otimes \sO_{Y_0}. 
$$
By Proposition \ref{propositiontrivial} these are trivial vector bundles, so
the rank of the map is constant by Remark \ref{remarkholomorphicconstant}.
Thus $\eqref{symmap}$ is surjective.

While the line bundle $L$ is globally generated and $\varphi$-big, 
the line bundle $L+mE$ is in general only $\varphi$-globally generated
on a general fibre. Hence the natural map
$$
\varphi^\star \varphi_\star (\sO_\Gamma(L+mE)) \rightarrow L+mE
$$
is not necessarily surjective. However, up to taking
a resolution of the indeterminacies $\Gamma' \rightarrow \Gamma$ by some smooth projective manifold $\Gamma'$, and replacing $L$ and $E$ by their pull-backs,
we can suppose without loss of generality that we have
$$
\varphi^\star \varphi_\star (\sO_\Gamma(L+mE)) \twoheadrightarrow M \hookrightarrow L+mE
$$ 
where $M$ is a line bundle. Pushing down to $Y$ gives morphisms
$$
\varphi_\star (\sO_\Gamma(L+mE)) \rightarrow \varphi_\star (\sO_\Gamma(M)) \hookrightarrow \varphi_\star (\sO_\Gamma(L+mE)) 
$$ 
and the composition is simply the identity. Thus the injection $\varphi_\star (\sO_\Gamma(M)) \hookrightarrow \varphi_\star (\sO_\Gamma(L+mE))$ is an isomorphism. In particular we obtain
\begin{equation} \label{mobilefix}
L+mE = M + N
\end{equation}
where $M$ is $\varphi$-globally generated and $N$ is an effective divisor
that does not surject onto $Y$. For a general $\varphi$-fibre $F$
we have 
$$
L|_F = (L+mE)|_F = (M + N)|_F = M|_F,
$$
hence $M|_F$ is very ample.

The surjective map $\varphi^\star \varphi_\star (\sO_\Gamma(M)) \twoheadrightarrow M$
defines a morphism $\Gamma \rightarrow \PP(\varphi_\star (\sO_\Gamma(M)))$,
and we denote by  $\Univ_M \subset \PP(\varphi_\star (\sO_\Gamma(M)))$ its image.
We summarise the situation and notation in a commutative diagram:
$$
\xymatrix{
\Gamma \ar[rrd]_\varphi \ar[rr]^{\psi_M} & & \Univ_M \ar[d]^{\varphi_M}  \ar @{^{(}->}[r] &   \PP(\varphi_\star (\sO_\Gamma(M))) \ar[ld]^{\pr} \\
& &
Y
}
$$ 

{\em Step 1. We prove that $\fibre{\varphi_M}{Y_0} \simeq Y_0 \times F$ with $F$ a general $\varphi$-fibre.}
For every $p \in \N$ the inclusion $pM \hookrightarrow p(L+mE)$ induces a commutative diagram
$$
\xymatrix{
\mbox{Sym}^p (\varphi_\star (\sO_\Gamma(M))) \otimes \sO_{Y_0} 
\ar[r] \ar[d]^{\simeq}
&
\varphi_\star (\sO_\Gamma(pM)) \otimes \sO_{Y_0} \ar @{^{(}->}[d]
\\
\mbox{Sym}^p (\varphi_\star (\sO_\Gamma(L+mE))) \otimes \sO_{Y_0} 
\ar@{->>}[r] 
&
\varphi_\star (\sO_\Gamma(p(L+mE))) \otimes \sO_{Y_0}
}
$$ 
The isomorphism in the left column and the surjectivity of the bottom line were established above,
so the right column is also an isomorphism.
If we set 
$$
G :=(\det \varphi_\star (\sO_\Gamma(L+mE)) \otimes \sO_{Y_0} )^{\otimes \frac{1}{r}} \simeq  (\det \varphi_\star (\sO_\Gamma(M)) \otimes \sO_{Y_0})^{\otimes \frac{1}{r}},
$$ 
we obtain by Proposition \ref{propositiontrivial} that 
\begin{equation}\label{ideovery}
\varphi_\star (\sO_\Gamma(pM)) \otimes \sO_{Y_0}  \simeq (G \otimes \sO_{Y_0})^{\oplus r_p}
\end{equation}
for every $p \in\mathbb{N}$.

Denote by $\sO_{\PP(\varphi_\star (\sO_\Gamma(M)))}(1)$ the tautological bundle on $\PP(\varphi_\star (\sO_\Gamma(M)))$,
and by $\mathcal{I}_{\Univ_M}$ the ideal sheaf of $\Univ_M \subset  \PP(\varphi_\star (\sO_\Gamma(M)))$.
For $p \in \N$ sufficiently high, the 
direct image $\pr_\star (\mathcal{I}_{\Univ_M}(p))$ is $\pr$-globally generated,
and the higher direct images $R^j \pr_\star (\mathcal{I}_{\Univ_M}(p))$ vanish.
Thus we have an exact sequence
\begin{equation} \label{pushidealsheaf}
0
\rightarrow \pr_\star (\mathcal{I}_{\Univ_M}(p))
\rightarrow \pr_\star (\sO_{\PP(\varphi_\star (\sO_\Gamma(M)))}(p)) \simeq \mbox{Sym}^p (\varphi_\star (\sO_\Gamma(M)))
\rightarrow (\varphi_M)_\star (\mathcal{O}_{\Univ_M} (p)) 
\rightarrow 
0,
\end{equation}
where $\mathcal{O}_{\Univ_M} (p)$ denotes the restriction of the tautological bundle to $\Univ_M$.
By construction we have $M = \psi_M ^\star \sO_{\PP(\varphi_\star (\sO_\Gamma(M)))}(1)$, so by the projection formula
there is a natural inclusion
\begin{equation} \label{notnormal}
(\varphi_M)_\star (\mathcal{O}_{\Univ_M}(p)) \rightarrow \varphi_\star (\sO_\Gamma(pM)).
\end{equation}
This morphism is an isomorphism in the generic point, since $\psi_M$ is an isomorphism on the general $\varphi$-fibre.
Twisting \eqref{pushidealsheaf} and \eqref{notnormal} with $G^{-p}$ we obtain morphisms
$$
\mbox{Sym}^p \varphi_\star (\sO_\Gamma(M))  \otimes G^{-p} 
\rightarrow
(\varphi_M)_\star (\mathcal{O}_{\Univ_M} (p)) \otimes G^{-p} 
\rightarrow 
\varphi_\star (\sO_\Gamma(pM)) \otimes G^{-p}.
$$
By \eqref{ideovery} the restriction of the first and the last sheaf to $Y_0$ are trivial vector bundles, so
by Remark \ref{remarkholomorphicconstant} the rank of the map is constant.
In particular the inclusion $(\varphi_M)_\star (\mathcal{O}_{\Univ_M}(p)) \rightarrow \varphi_\star (\sO_\Gamma(pM))$
is an isomorphism on $Y_0$. Hence
$$
(\varphi_M)_\star (\mathcal{O}_{\Univ_M}(p)) \otimes G^{-p} \otimes \sO_{Y_0}
$$
is a trivial vector bundle, and by the exact sequence \eqref{pushidealsheaf}
$$
\pr_\star (\mathcal{I}_{\Univ_M}(p)) \otimes G^{-p} \otimes \sO_{Y_0}
$$
is also trivial. Finally, this proves that the equations defining $\Univ_M$ in the product 
$\PP(\varphi_\star (\sO_\Gamma(M) \otimes \sO_{Y_0}) \simeq Y_0 \times \PP^{r-1}$ are constant,
hence $\Univ_M$ is a product.

{\em Step 2. We prove that the tangent bundle of $\fibre{\varphi}{Y_0} \cap (\Gamma \setminus E)$ splits.}
Recall that by \eqref{mobilefix} we have $L+mE = M + N$ with $N$ an effective divisor corresponding to 
$\varphi$-fixed components of the $L+mE$.
The line bundle $L$ is $\varphi$-globally generated, so any relative fixed component of $L+mE$ is 
is contained in $mE$. Thus $mE - N$ is an effective divisor.

Now observe the following: given any fibre $\fibre{\varphi}{y}$, let $F_0$ 
be an irreducible component of $\fibre{\varphi}{y}$ that is not contained in $E$. 
Then $F_0$ maps birationally 
onto its image in $X$, so $L|_{F_0} = (\pi^\star A)|_{F_0}$ is nef and big. Since 
$$
M|_{F_0} = L|_{F_0} + (mE-N)|_{F_0}
$$
we obtain that $M|_{F_0}$ is also big. Thus $F_0$ is not contracted by $\psi_M$.

Consider now the birational morphism 
$$
\psi_M^0: \fibre{\varphi}{Y_0} \rightarrow \fibre{\varphi_M}{Y_0} \simeq Y_0 \times F.
$$
Since $Y_0 \times F$ is smooth, the $\psi_M^0$-exceptional locus has pure codimension one, and we denote by
$D$ one of its irreducible components. Since $\fibre{\varphi}{Y_0} \rightarrow Y_0$ is flat, the image $\varphi(D)$
is a divisor, and we denote by $F_0$ an irreducible component of a general fibre of $D \rightarrow \varphi(D)$.
If $F_0 \not\subset E$, then we have just shown that $F_0$ is not contracted by $\psi_M$. Thus $D$ would not be
contracted by $\psi_M^0$. This shows that the exceptional locus of $\psi_M^0$ is contained in $E$.
In particular 
\begin{equation} \label{opensplitting}
T_{\fibre{\varphi}{Y_0} \cap (\Gamma \setminus E)} \simeq V_1^0 \oplus V_2^0,
\end{equation}
where $V_1^0$ (resp. $V_2^0$) is the pull-back of $T_{Y_0 \times F/Y_0}$ (resp. $T_{Y_0 \times F/F}$)
under the embedding 
$$
\fibre{\varphi}{Y_0} \cap \Gamma \setminus E \hookrightarrow \fibre{\varphi_M}{Y_0} \simeq Y_0 \times F.
$$
Note that $V_1^0$ and $V_2^0$ are both algebraically integrable foliations and the general leaves of $V_1^0$ 
are $\varphi$-fibres.

{\em Step 3. Splitting of $T_X$ and conclusion.}
Since $\Gamma \setminus E \rightarrow X$ is an embedding, we can identify
$\fibre{\varphi}{Y_0} \cap (\Gamma \setminus E)$ to a Zariski open subset $X_0$ of $X$, and we denote by
$j: X_0 \hookrightarrow X$
the inclusion. Then
$$
V_1 := (j_\star(V_1^0))^{**}, \qquad V_2 := (j_\star(V_2^0))^{**}
$$
are reflexive sheaves on $X$. By \eqref{opensplitting} we have
$$
(V_1 \oplus V_2) \otimes \sO_{X_0} = T_{X_0}.
$$
By Proposition \ref{propositionfundamentals},2) the complement of $X \setminus X_0$
has codimension at least two.
Thus the equality above extends by Proposition \ref{facts},3) to an isomorphism of reflexive sheaves
$$
V_1 \oplus V_2 \simeq T_X.
$$
Yet this implies that the upper semicontinuous functions $\dim \left( V_{j,x} \otimes_{\sO_{X,x}} \C(x) \right)$ are constant,
since their sum is equal to $\rk T_X$. Thus $V_1$ and $V_2$ are locally free \cite[II,Ex.5.8.c)]{Har77}
and, identifying $V_1$ and $V_2$ to their images in $T_X$, they define regular foliations on $X$.
We have seen above that a general $V_1$-leaf is a general fibre of the MRC-fibration, hence rationally connected. 
By \cite[Cor.2.11]{a1} this implies that there exists a smooth morphism $X \rightarrow Y$ such
that $T_{X/Y}=V_1$. Since $T_X = V_1 \oplus V_2$, the relative tangent sequence of this submersion is split by $V_2$.
Since $V_2$ is integrable, the classical Ehresmann theorem (e.g. \cite[V.,\S 2,Prop.1 and Thm.3]{CLN85}) implies that
the universal cover of $X$ is a product. Yet $X$ is simply connected, so $X$ itself is a product.
$\square$

\begin{proof}[Proof of Theorem \ref{theoremMRC}]

If $X$ is not uniruled, then $K_X \equiv 0$ by \cite{BDPP13}. Thus we can suppose that $X$ is uniruled, and denote
by $\sF \subset T_X$ the unique integrable saturated subsheaf such that for a very general point $x \in X$, the $\sF$-leaf
through $x$ is a fibre of the MRC-fibration. Note that a saturated subsheaf is reflexive.

{\em Step 1. We prove that $\sF \subset T_X$ is a subbundle.}
Note first that if $\holom{f}{X'}{X}$ is a finite \'etale cover, then, since $f^\star$ is exact for flat morphisms, the pull-back $f^\star \sF \subset f^\star T_X = T_{X'}$ is the unique integrable saturated subsheaf corresponding to the MRC-fibration of $X'$. Moreover if $f^\star \sF \subset T_{X'}$
is a subbundle, then $\sF \subset T_X$ is a subbundle. Thus the claim is invariant under finite \'etale covers,
in particular it holds for manifolds with finite fundamental group by Theorem \ref{theoremmain}.

By \cite[Cor.1.4]{Cao16} we know that, maybe after finite \'etale cover, we have a locally trivial Albanese map
$\alpha: X \rightarrow \mbox{Alb}(X)$ such that the fibre $G$ has finite fundamental group. Since $\Alb(X)$ is not uniruled,
we have a factorisation
$$
\sF  \subset T_{X/\mbox{Alb}(X)}  \subset T_X.
$$
Fix any point $y \in \mbox{Alb}(X)$ and choose
an analytic neighborhood $U \subset \mbox{Alb}(X)$ such that $\fibre{\alpha}{U} \simeq U \times G$.
Since $G$ has finite fundamental group, we already know 
that the MRC-fibration can be represented by $U \times G \rightarrow U \times Z$, where $G \rightarrow Z$ is a smooth fibration.
By uniqueness of $\sF|_{\fibre{\alpha}{U}}$ this shows that $\sF|_{\fibre{\alpha}{U}} \subset T_X|_{\fibre{\alpha}{U}}$
is a subbundle.

{\em Step 2. Structure of MRC-fibration.} By Step 1 we know that $\sF \subset T_X$ is a regular foliation.
Since it has a rationally connected leaf, we know by \cite[Cor.2.11]{a1} 
that there exists a smooth morphism $\varphi: X \rightarrow Y$ such that $T_{X/Y}=\sF$. Since $-K_X$ is nef and $\varphi$ is smooth,
we know that $-K_Y$ is nef \cite[Cor.3.15]{Deb01}. Since $Y$ is not uniruled \cite{GHS03}, this proves that $K_Y \equiv 0$.

We claim that $\varphi$ is locally trivial. By the Fischer-Grauert theorem \cite{FG65} it is enough to show that all the fibres are biholomorphic. Note again that this statement is invariant under \'etale cover: rationally connected manifolds are simply connected, so
a $\varphi$-fibre lifts to a $\varphi'$-fibre if $X' \rightarrow X$ is \'etale. Thus by 
\cite[Cor.1.4]{Cao16} we reduce to the case where we have a locally trivial Albanese map
$\alpha: X \rightarrow \mbox{Alb}(X)$ such that the fibre $G$ has finite fundamental group.
By invariance under \'etale covers and Theorem \ref{theoremmain} we know that the MRC-fibration of $G$ is locally trivial,
so $\varphi$-fibres being contained in the same $\alpha$-fibre are biholomorphic. Yet $\alpha$ is locally trivial, and automorphisms
preserve the MRC-fibration, so all the $\varphi$-fibres are biholomorphic.
\end{proof}

\begin{appendix}
\section{Some analytic results}

To begin with, we first prove an approximation lemma which is a direct consequence  of \cite[thm 1.1]{Dem92} 
and a standard gluing lemma \cite[section 3]{Dem92}.

\begin{lemma}\label{approx}
Let $C$ be a $1$-dimensional projective manifold and let $E$ be a vector bundle of rank $n$ on $C$. 
Let $\pi : \mathbb{P} (E) \rightarrow C$ be the natural projection. 
Let $T =\alpha +dd^c \psi \geq 0$ be a positive current of analytic singularity, where $\alpha$ is smooth form and $\psi$ is a quasi-psh function.
We suppose moreover that $\psi$  is smooth over the generic fiber of $\pi$.

Then for every $c>0$, there exists a sequence of closed almost positive $(1,1)$-currents $T_{c , k} =\alpha +dd^c \psi_{c, k}$ 
such that $\psi_{c,k}$ is smooth on $\mathbb{P} (E) \setminus E_c (T)$, decreases to $\psi$ as $k$ tends to $+\infty$, and
$$T_{c , k} \geq -\ep_k \omega ,$$
where $\lim\limits_{k\rightarrow +\infty}\ep_k  = 0$ and $\nu (T_{c , k}, x) = (\nu (T, x) -c )_+$ at every point $x\in X$. 
Here $E_c (T) :=\{x | \text{ } \nu (T,x ) \geq c \}$ and $(\nu (T, x) -c )_+ :=\max \{\nu (T, x) -c, 0\}$.
\end{lemma}

\begin{proof}
As $\psi$  is smooth over the generic fiber of $\pi$, we can take a Stein cover $\{ U_i\}$ of $C$ such that $T$ is smooth over $\pi^{-1} (U_i \cap U_j)$ for every $i \neq j$.
Over $\pi^{-1} (U_i)$, since $T_{\mathbb{P} (E) / C}$ is relatively ample, we can find a smooth metric on $\mathcal{O} _{\mathbb{P}  (T_{\pi^{-1} (U_i)})} (1)$ such that the curvature form is positive.
By applying  \cite[thm 1.1]{Dem92}, we can find a sequence of closed almost positive $(1,1)$-currents $T_{c , k, i} =\alpha +dd^c \psi_{c,k,i}$ on $\pi^{-1} (U_i)$, such that $\psi_{c,k,i}$ is smooth on $\pi^{-1} (U_i)\setminus E_c (T)$, decreases to $\psi |_{\pi^{-1} (U_i)}$ as $k$ tends to $+\infty$, and
$$T_{c , k, i} \geq -\ep_k \omega \qquad\text{on }\pi^{-1} (U_i) ,$$
where $\ep_k \rightarrow 0$ and $\nu (T_{c , k,i}, x) = (\nu (T, x) -c )_+$ at every point $x\in \pi^{-1} (U_i)$.

\medskip

Since $\psi$ is smooth over $\pi^{-1} (U_i \cap U_j)$ for every $i \neq j$, for any set $V \Subset  U_i \cap U_j$, we know that 
$\lim\limits_{k\rightarrow +\infty} \|\psi_{c,k,i} -\psi_{c,k,j}\|_{\mathcal{C}^0 (\pi^{-1}(V))} =0$. 
By applying the standard gluing process \cite[lemma 3.5]{Dem92}, we can glue $\{T_{c , k,i}\}_i$ together by losing a bit of positivity.
The lemma is thus proved.
\end{proof}

Thanks to Lemma \ref{approx}, we can get an upper estimate of the Lelong numbers of $T$ by using the same arguments in \cite[thm 1.7]{Dem92}.
\begin{proposition}\label{lelongnumcont}
In the setting of Lemma \ref{approx}, set $a := \max_{x\in \mathbb{P} (E)} \nu (T,x)$. 
Then there exists a $k\in \{1, \cdots , n\}$, a closed $(k,k)$-positive current $\Theta_k$ in 
the same class of $( c_1 (T))^k \in H^{k,k} (\mathbb{P} (E))$ 
and a $k$-codimensional subvariety $Z$ contained in some fiber of $\pi$ such that 
$$
\Theta_k \geq (\frac{a}{n})^k \cdot [Z] 
$$
in the sense of current.
\end{proposition}

\begin{proof}
Thanks to Lemma \ref{approx}, the proposition can be proved by the same argument as in \cite[thm 1.7]{Dem92}. For the convenience of readers, we give the proof here.

We follow the notations introduced \cite[thm 1.7]{Dem92}: let $0=b_1 \leq b_2 \cdots \leq b_n \leq b_{n+1}$ 
be the sequence of
''jumping values'' $b_p$, namely $\codim E_c (T) =p$ when $c\in (b_p, b_{p+1} ]$.
Here, $\codim E_c (T) =p$ means that all components of $E_c (T)$ is of codimension  $\geq p$ and at least some component
of  $E_c (T)$ is of codimension $p$.
By the definition of $a$, there exists a $b_k$ such that $b_{k+1} - b_k \geq \frac{a}{n}$.
Let $Z$ be a $k$-codimensional component of $E_{b_{k+1}} (T)$ such that the generic Lelong of $T$ along $Z$ is $b_{k+1}$.
Since $\psi$ is smooth on the generic fiber of $\pi$, $Z$ is contained in some fiber of $\pi$.

\medskip

We now prove by induction the following claim: for every $1 \leq p \leq k$, we can find $(p,p)$-positive current $\Theta_p$ in 
the same class of $(c_1 (T))^k$ such that 
\begin{equation}\label{lelongnum}
 \nu (\Theta_p , x) \geq (b_{k+1} -b_1) \cdots (b_{k+1} -b_p) \qquad\text{for every }x\in Z.
\end{equation}

For $p=1$, we can take $\Theta_1 = T$ and \eqref{lelongnum} holds. Now, suppose that $\Theta_{p-1}$ has been constructed. For $c > b_p$, the current $T_{c,k}$ 
produced by Lemma \ref{approx} satisfies $T_{c,k} \geq -\ep_k \omega$ and the singular locus of $T_{c,k}$ is contained in $E_c (T)$
which is of codimension at least $p$.
Then \cite[Lemma 7.4]{Dem92} shows that $\Theta_{p-1} \wedge (T_{c,k} +\ep_k\omega)$ is a well defined positive current.
We define 
$$\Theta_p =\lim_{c \rightarrow b_p +0} \lim_{k\rightarrow +\infty} \Theta_{p-1} \wedge (T_{c,k} +\ep_k\omega)$$
possibly after extracting some weakly convergent subsequence. Moreover, we have
$$\nu (\Theta_p , x) \geq \nu (\Theta_{p-1} , x) \times \varlimsup_{c\rightarrow b_p +0} \varlimsup_{k\rightarrow +\infty} \nu (T_{c,k} ,x) \geq 
 \nu (\Theta_{p-1} , x) \times (\nu (T, x) -b_p)_+ .$$
In particular, for every $x\in Z$, we have $ \nu (\Theta_p , x) \geq (b_{k+1} -b_1) \cdots (b_{k+1} -b_p)$.  \eqref{lelongnum} is proved.

\medskip

We now prove the proposition. Thanks to \eqref{lelongnum}, we can find a closed $(k,k)$-positive current $\Theta_k$ in 
the same class of $(c_1 (T))^k$ such that 
$$
\nu (\Theta_k ,x) \geq (b_{k+1}-b_1)\cdots (b_{k+1} -b_k) \geq  (b_{k+1} -b_k)^k  \geq (\frac{a}{n})^k  \qquad\text{for every }x\in Z .
$$
As $Z$ is of codimension $k$, Siu's decomposition theorem implies 
$$\Theta_k \geq (\frac{a}{n})^k \cdot [Z]$$
in the sense of current. The proposition is proved.
\end{proof}
\end{appendix}


\end{document}